\documentclass[letterpaper,12pt,draftcls,onecolumn,romanappendices]{IEEEtran}

\usepackage{amsmath,amssymb,theorem}
\usepackage{graphicx}
\usepackage{algorithm,algorithmic}
\usepackage{psfrag}
\usepackage{url}
	
\usepackage{color,xspace}
\usepackage{subfigure}

\newtheorem{theorem}{Theorem}[section]
\newtheorem{lemma}[theorem]{Lemma}

\newtheorem{remark}[theorem]{Remark}
\newtheorem{example}[theorem]{Example}	
\newtheorem{corollary}[theorem]{Corollary}
\newtheorem{proposition}[theorem]{Proposition}
\newtheorem{definition}[theorem]{Definition}

\newcommand{\real}{{\mathbb{R}}}

\newcommand{\realnonnegative}{\mathbb{R}_{\ge 0}}
\newcommand{\integernonnegative}{\mathbb{Z}_{\ge 0}}

\newcommand{\GG}{{\mathcal{G}}}

\newcommand{\NN}{{\mathcal{N}}}

\newcommand{\UU}{{\mathcal{U}}}

\newcommand{\XX}{{\mathcal{X}}}

\newcommand{\unit}[1]{\operatorname{unit}(#1)}

\newcommand{\loc}[1]{\operatorname{loc}(#1)}

\newcommand{\diag}[1]{\operatorname{diag}\left( #1\right)}

\newcommand{\grad}{\nabla}

\newcommand{\vertices}{V}
\newcommand{\edges}{E}

\newcommand{\vmax}{v_{\text{max}}}
\newcommand{\umax}{u_{\text{max}}}

\renewcommand{\epsilon}{\varepsilon}

\newcommand{\reachable}{\mathcal{R}}
\newcommand{\powerset}{\mathbb{P}^\text{cc}}
\newcommand{\powersetc}{\mathbb{P}^\text{c}}
\renewcommand{\time}[1]{t_{\text{#1}}}

\newcommand{\until}[1]{\{1,\dots, #1\}}

\newcommand{\map}[3]{#1: #2 \rightarrow #3}

\newcommand{\setdef}[2]{\{#1 \; | \; #2\}}

\newcommand{\cball}[2]{\overline{B}(#1,#2)}

\newcommand{\TwoNorm}[1]{\|#1\|_2}

\renewcommand{\hat}{\widehat}

\newcommand{\commgraph}{\GG_\text{comm}}

\renewcommand{\commgraph}{\GG}

\newcommand{\uself}{u^\text{self}}
\newcommand{\uteam}{u^\text{team}}

\newcommand{\safecontrol}{u^\text{sf}}

\newcommand{\availablestates}{x_{\NN}}
\newcommand{\estimatestates}{\hat{x}_{\NN}}

\newcommand{\promisesets}{X_{\NN}}
\newcommand{\guaranteedsets}{\mathbf{X}_\NN}
\newcommand{\gset}{\mathbf{X}}

\newcommand{\delay}{\Delta}
\newcommand{\maxdelay}{\bar{\Delta}}
\newcommand{\maxnoise}{\bar{\omega}}
\newcommand{\si}{\mathcal{L}_iV^\text{sup}}

\newcommand{\dwellself}{T_{\text{d,self}}}
\newcommand{\dwellevent}{T_{\text{d,event}}}
\newcommand{\expirationevent}{T_{\text{exp}}}

\newcommand{\hausdorff}{d_H}
\newcommand{\functiondist}{d_{\text{func}}}

\newcommand{\staterule}{R^\text{s}}

\newcommand{\staticrule}{R^\text{cb}}

\renewcommand{\loc}{\operatorname{loc}}

\newcommand{\algomap}{M}
\newcommand{\timeschedule}{\mathcal{T}}

\newcommand{\oprocendsymbol}{\hbox{$\bullet$}}
\newcommand{\oprocend}{\relax\ifmmode\else\unskip\hfill\fi\oprocendsymbol}
\def\eqoprocend{\tag*{$\bullet$}}

\newcommand{\longthmtitle}[1]{\mbox{}\textup{\textbf{(#1)}}}

\newcommand{\myclearpage}{\clearpage}
\renewcommand{\myclearpage}{}

\newcommand{\algoselfrequests}{\textsc{self-triggered information
    updates}\xspace} 

\newcommand{\algosocial}{\textsc{team-triggered law}\xspace}
\newcommand{\algosocialdelays}{\textsc{robust team-triggered law}\xspace}

\begin{document}

\title{Team-triggered coordination for real-time
  \\
  control of networked cyber-physical systems\thanks{Preliminary
    versions of this paper have appeared as~\cite{CN-JC:13-acc}
    and~\cite{CN-JC:13-sv}.}}


\author{Cameron Nowzari \qquad Jorge Cort\'es \thanks{The authors are
    with the Department of Mechanical and Aerospace Engineering,
    University of California, San Diego, California, USA, {\tt\small
      \{cnowzari,cortes\}@ucsd.edu}}}

\maketitle

\begin{abstract}
  This paper studies the real-time implementation of distributed
  controllers on networked cyber-physical systems.  We build on the
  strengths of event- and self-triggered control to synthesize a
  unified approach, termed team-triggered, where agents make promises
  to one another about their future states and are responsible for
  warning each other if they later decide to break them. The
  information provided by these promises allows individual agents to
  autonomously schedule information requests in the future and sets
  the basis for maintaining desired levels of performance at lower
  implementation cost.  We establish provably correct guarantees for
  the distributed strategies that result from the proposed approach
  and examine their robustness against delays, packet drops, and
  communication noise.  The results are illustrated in simulations of
  a multi-agent formation control problem.
\end{abstract}

\section{Introduction}

A growing body of work studies the design and real-time implementation
of distributed controllers to ensure the efficient and robust
operation of networked cyber-physical systems.  In multi-agent
scenarios, energy consumption is correlated with the rate at which
sensors take samples, processors recompute control inputs, actuator
signals are transmitted, and receivers are left on listening for
potential incoming signals.  Performing these tasks periodically is
costly, might lead to inefficient implementations, or face hard
physical constraints.
To address these issues, the goal of triggered control is to identify
criteria that allow agents to tune the implementation of controllers
and sampling schemes to the execution of the task at hand and the
desired level of performance.  In event-triggered control, the focus
is on detecting events during the network execution that are relevant
from the point of view of task completion and should trigger specific
agent actions.  In self-triggered control, the emphasis is instead on
developing tests that rely only on current information available to
individual agents to schedule future actions. 
Event-triggered strategies generally result in less samples or controller 
updates but, when executed over networked systems, may 
be costly to implement because of the need for
continuous availability of the information required to check the
triggers. Self-triggered strategies are more easily amenable to
distributed implementation but result in conservative executions
because of the over-approximation by individual agents about the state
of the environment and the network.  These strategies might be also
beneficial in scenarios where leaving receivers on to listen to
potential messages is costly.  Our objective in this paper is to build
on the strengths of event- and self-triggered control to synthesize a
unified approach for controlling networked systems in real time that
combines the best of both worlds.

\subsubsection*{Literature review}

The need for systems integration and the importance of bridging the
gap between computing, communication, and control in the study of
cyber-physical systems cannot be
overemphasized~\cite{KDK-PRK:12,JS-XK-GK-NK-PA-VG-BG-JB-SW:12}.
Real-time controller implementation is an area of extensive research
including periodic~\cite{DH-WL:05,KJA-BW:96}, 
event-triggered~\cite{PW-MDL:09,KJA-BMB:02,PT:07,WPMHH-JHS-PPJB:08}, 
and self-triggered~\cite{MV-PM-JMF:03,RS-FF:06,AA-PT:10} 
procedures. Our approach shares with these works the aim of trading
computation and decision making for less communication, sensor, or
actuator effort while still guaranteeing a desired level of
performance.  Of particular relevance to this paper are works that
study self- and event-triggered implementations of controllers for
networked cyber-physical systems.  The predominant paradigm is that of
a single plant that is stabilized through a decentralized triggered
controller over a sensor-actuator network, see
e.g.~\cite{MMJ-PT:10,XW-NH:10,MCFD-WPMH:12}. Fewer works have
considered scenarios where multiple plants or agents together are the
subject of the overall control design. Exceptions include consensus
via
event-triggered~\cite{DVD-EF-KHJ:12,GS-KHJ:11,XM-TC:13} 
or self-triggered control~\cite{DVD-EF-KHJ:12,MMJ-PT:08},
rendezvous~\cite{YF-GF-YW-CS:13}, model predictive
control~\cite{AE-DVD-KJK:11}, and model-based event-triggered
control~\cite{EG-PJA:13,WPMHH-MCFD:13}.  The event-triggered
controller designed in~\cite{DVD-EF-KHJ:12} for a decentralized system
with multiple plants requires agents to have continuous information
about each others' states.  The works
in~\cite{DVD-EF-KHJ:12,CN-JC:11-auto} implement self-triggered
communication schemes to perform distributed control where agents
assume worst-case conditions for other agents when deciding when new
information should be obtained.  Distributed strategies based on
event-triggered communication and control are explored
in~\cite{XW-MDL:08}, where each agent has an a priori computed local
error tolerance and once it violates it, the agent broadcasts its
updated state to its neighbors.  The same event-triggered approach is
taken in~\cite{MZ-CGC:10} to implement gradient control laws that
achieve distributed optimization.  The
works~\cite{EG-PJA:13,XW-MDL:11,GSS-DVD-KHJ:13} are closer in spirit
to the ideas presented here.  In the interconnected system considered
in~\cite{EG-PJA:13}, each subsystem helps neighboring subsystems by
monitoring their estimates and ensuring that they stay within some
performance bounds. The approach requires different subsystems to have
synchronized estimates of one another even though they do not
communicate at all times.  In~\cite{XW-MDL:11,GSS-DVD-KHJ:13}, agents
do not have continuous availability of information from neighbors and
instead decide when to broadcast new information to them.

\paragraph*{Statement of contributions}

We propose a novel scheme for the real-time control of networked
cyber-physical systems that combines ideas from event- and
self-triggered control.  Our approach is based on agents making
promises to one another about their future states and being
responsible for warning each other if they later decide to break
them. This is reminiscent of event-triggered implementations.
Promises can be broad, from tight state trajectories to loose
descriptions of reachability sets.  With the information provided by
promises, individual agents can autonomously determine when in the
future fresh information will be needed to maintain a desired level of
performance. This is reminiscent of self-triggered implementations.
The benefits of the proposed scheme are threefold. First, because of
the availability of the promises, agents do not require continuous
state information about neighbors, in contrast to event-triggered
strategies implemented over distributed systems that require the
continuous availability of the information necessary to check the
relevant triggers.  Second, because of the extra information provided
by promises about what other agents plan to do, agents can generally
wait longer periods of time before requesting new information and
operate more efficiently than if only worst-case scenarios are
assumed, as is done in self-triggered control. Less overall
communication is beneficial in reducing the total network load and
decreasing chances of communication delays or packet drops due to
network congestion. Lastly, we provide theoretical guarantees for the
correctness and performance of team-triggered strategies implemented
over distributed networked systems. Our technical approach makes use
of set-valued analysis, invariance sets, and Lyapunov stability.  We
also show that, in the presence of physical sources of error and under
the assumption that 1-bit messages can be sent reliably with
negligible delay, the team-triggered approach can be slightly modified
to be robust to delays, packet drops, and communication noise.
Interestingly, the self-triggered approach can be seen as a particular
case of the team-triggered approach where promises among agents simply
consist of their reachability sets (and hence do not actually
constrain their state).  We illustrate the convergence and robustness
results through simulation in a multi-agent formation control problem,
paying special attention to the implementation costs and the role of
the tightness of promises in the algorithm performance.


\paragraph*{Organization}
Section~\ref{se:statement} lays out the problem of interest.
Section~\ref{se:triggered-control} briefly reviews current real-time
implementation approaches based on agent triggers.
Section~\ref{se:team-triggered} presents the team-triggered approach
for networked cyber-physical systems.  Sections~\ref{se:analysis}
and~\ref{se:robustness} analyze the correctness and robustness,
respectively, of team-triggered strategies.  Simulations illustrate
our results in Section~\ref{se:simulations}.  Finally,
Section~\ref{se:conclusions} gathers our conclusions and ideas for
future work.

\paragraph*{Notation}
We let $\real$, $\realnonnegative$, and $\integernonnegative$ denote
the sets of real, nonnegative real, and nonnegative integer numbers,
respectively. The two-norm of a vector is $\TwoNorm{\cdot}$.  Given $x
\in \real^d$ and $\delta \in \realnonnegative$, $\cball{x}{\delta}$
denotes the closed ball centered at $x$ with radius $\delta$.  For
$A_i \in \real^{m_i \times n_i}$ with $i \in \until{N}$, we denote by
$\diag{A_1, \dots, A_N} \in \real^{m \times n}$ the block-diagonal
matrix with $A_1$ through $A_N$ on the diagonal, where $m=\sum_{i=1}^N
m_i$ and $n=\sum_{i=1}^N n_i$.  Given a set $S$, we denote by $|S|$
its cardinality. We let~$\powersetc(S)$, respectively~$\powerset(S)$,
denote the collection of compact, respectively, compact and connected,
subsets of~$S$. The Hausdorff distance between $S_1, S_2 \subset
\real^d$ is
\begin{align*}
  \hausdorff(S_1,S_2) = \max \{ \sup_{x \in S_1} \inf_{y \in S_2}
  \TwoNorm{x - y}, \sup_{y \in S_2} \inf_{x \in S_1} \TwoNorm{x - y}
  \} .
\end{align*}
The Hausdorff distance is a metric on the set of all non-empty compact
subsets of $\real^d$. Given two bounded set-valued functions $C_1, C_2
\in \mathcal{C}^0( I \subset \real ; \powersetc(\real^d) )$, its
distance is
\begin{align}\label{eq:dist-func}
  \functiondist(C_1,C_2) = \sup_{t \in I} \hausdorff( C_1(t), C_2(t) ) .
\end{align}
An undirected graph~$\commgraph = (\vertices, \edges)$ is a pair
consisting of a set of vertices $\vertices = \until{N}$ and a set of
edges $\edges \subset \vertices \times \vertices$ such that if $(i, j)
\in \edges$, then $(j, i) \in \edges$. The set of neighbors of a
vertex~$i$ is $\NN(i) = \setdef{j \in \vertices}{(i, j) \in \edges}$.
Given $v \in \prod_{i=1}^N \real^{n_i}$, we let $v^i_\NN = ( v_i, \{
v_j \}_{j \in \NN(i)})$ denote the components of $v$ that correspond
to vertex $i$ and its neighbors in~$\commgraph$.

\myclearpage
\section{Network modeling and problem statement}\label{se:statement}

We consider a distributed control problem carried out over an
unreliable wireless network. Consider $N$ agents whose communication
topology is described by an undirected graph~$\commgraph$. The fact
that $(i, j)$ belongs to $\edges$ models the ability of agents $i$ and
$j$ to communicate with one another. The agents~$i$ can communicate
with are its neighbors $\NN(i)$ in~$\commgraph$.  The state of~$i \in
\until{N}$, denoted $x_i$, belongs to a closed set $ \XX_i \subset
\real^{n_i}$.  The network state $x = \left( x_1 ,\dots ,x_N \right)$
therefore belongs to $\XX = \prod_{i=1}^N \XX_i$.
According to the discussion above, agent $i$ can access
$\availablestates^i $ when it communicates with its neighbors.  By
assumption, each agent has access to its own state at all times.  We
consider linear dynamics for each $i \in \until{N}$,
\begin{align}\label{eq:agentdynamics}
  \dot{x}_i = f_i(x_i, u_i) = A_i x_i + B_i u_i,
\end{align}
with $A_i \in \real^{n_i \times n_i}$, $B_i \in \real^{n_i \times
  m_i}$, and $u_i \in \UU_i $. Here, $\UU_i \subset \real^{m_i}$ is a
closed set of allowable controls for agent~$i$.  
We assume the existence of a \emph{safe-mode} controller
$\map{\safecontrol_i}{\XX_i}{\UU_i}$,
\begin{align*}
  A_i x_i + B_i \safecontrol_i(x_i) = 0 , \quad \text{for all } x_i
  \in \XX_i,
\end{align*}
i.e., a controller able to keep agent $i$'s state fixed.  The
existence of a safe-mode controller for a general controlled system
may seem restrictive, but there exist many cases, including nonlinear
systems, that admit one, such as single integrators or vehicles with
unicycle
dynamics. 
Letting $u = \left( u_1, \dots, u_N \right) \in \UU = \prod_{i=1}^N
\UU_i$, the dynamics can be described by
\begin{align}\label{eq:networkdynamics}
  \dot{x} = 
  A x + B u,
\end{align}
with $A = \diag{ A_1, \dots, A_N } \in \real^{n \times n}$ and $B =
\diag{ B_1, \dots, B_N } \in \real^{n \times m}$, where $n =
\sum_{i=1}^N n_i$, and $m = \sum_{i=1}^N m_i$.
We refer to the team of agents with communication
topology~$\commgraph$ and dynamics~\eqref{eq:networkdynamics}, where
each agent has a safe-mode controller and access to its own state at
all times, as a \emph{networked cyber-physical system}.  The goal is
to drive the agents' states to some desired closed set of
configurations $D \subset \XX$ and ensure that it stays there.
Depending on how $D$ is defined, this objective can capture different
coordination tasks, including deployment, rendezvous, and formation
control. The goal of the paper is not to design the controller that
achieves this but rather synthesize efficient strategies for the
real-time implementation of a given controller.

Given the agent dynamics, the communication graph $\commgraph$, and
the set $D$, our starting point is the availability of a control law
that drives the system asymptotically to $D$. Formally, we assume that
a continuous map~$\map{u^*}{\XX}{\UU}$ and a continuously
differentiable function $\map{V}{\XX}{\real}$, bounded from below
exist such that $D$ is the set of minimizers of $V$ and, for all $x
\notin D$,
\begin{subequations}\label{eq:derivativeassumption}
  \begin{align}
    \grad_i V(x) \left( A_i x_i + B_i u_i^*(x) \right) & \leq 0 ,
    \quad i \in \until{N},
    \label{eq:localassumption}
    \\
    \sum_{i = 1}^N \grad_i V(x) \left( A_i x_i + B_i u_i^*(x) \right)
    & < 0 . \label{eq:desiredderivative}
  \end{align}
\end{subequations}
We assume that both the control law $u^*$ and the gradient $\nabla V$
are distributed over $\commgraph$. By this we mean that, for each $i
\in \until{N}$, the $i$th component of each of these objects only
depends on $\availablestates^{i}$, rather than on the full network
state~$x$. For simplicity, and with a slight abuse of notation, we
write $u_i^*( \availablestates^{i}) \in \UU_i$ and $\grad_i
V(\availablestates^{i}) \in \real^{n_i}$ to emphasize this fact when
convenient. This property has the important consequence that agent $i$
can compute these quantities with the exact information it can obtain
through communication on~$\commgraph$.  


\begin{remark}\longthmtitle{Assumption on non-negative contribution of
    each agent to task completion} {\rm Note
    that~\eqref{eq:desiredderivative} simply states that $V$ is a
    Lyapunov function for the closed-loop
    system. Instead,~\eqref{eq:localassumption} is a more restrictive
    assumption that essentially states that each agent does not
    individually contribute in a negative way to the evolution of the
    Lyapunov function.  This latter assumption can in turn be
    relaxed~\cite{MMJ-PT:10} by selecting parameters
    $\alpha_1,\dots,\alpha_N \in \real$ with $\sum_{i=1}^N \alpha_i =
    0$ (note that some $\alpha_i$ would be positive and others
    negative) and specifying instead that, for each $i \in \until{N}$,
    the left-hand side of~\eqref{eq:localassumption} should be less
    than or equal to~$\alpha_i$. Along these lines, one could envision
    the design of distributed mechanisms to dynamically adjust these
    parameters, but we do not go into details here for space reasons.
    \oprocend}
\end{remark}

From an implementation viewpoint, the controller $u^*$ requires
continuous agent-to-agent communication and continuous updates of the
actuator signals, making it unfeasible for practical scenarios. In the
following section we develop a self-triggered communication and
control strategy to address the issue of selecting time instants for
information sharing.


\myclearpage
\section{Self-triggered communication and control}\label{se:triggered-control}

This section provides an overview of the self-triggered communication
and control approach to solve the problem described in
Section~\ref{se:statement}. In doing so, we also introduce several
concepts that play an important role in our discussion later.  The
general idea 
is to guarantee that the time derivative of the Lyapunov function~$V$
along the trajectories of the networked cyber-physical
system~\eqref{eq:networkdynamics} is less than or equal to $0$ at all
times, even when the information used by the agents is inexact.

To model the case that agents do not have perfect information about
each other at all times, we let each agent $i \in \until{N}$ keep an
estimate $\hat{x}_j^i$ of the state of each of its neighbors $j \in
\NN(i)$. Since $i$ always has access to its own state,
$\estimatestates^i(t) = ( x_i(t) , \{ \hat{x}_j^i(t) \}_{j \in
  \NN(i)})$ is the information available to agent $i$ at time $t$.
Since agents do not have access to exact information at all times,
they cannot implement the controller $u^*$ exactly, but instead use
the feedback law
\begin{align*}
  \uself_i(t) = u_i^*(\estimatestates^i(t)).
\end{align*}
We are now interested in designing a triggering method such that agent
$i$ can decide when $\estimatestates^i(t)$ needs to be updated. Let
$\time{last}$ be the last time at which all agents have received
information from their neighbors.  Then, the time $\time{next}$ at
which the estimates should be updated is when
\begin{align}\label{eq:eventnetwork}
  \frac{d}{dt} V (x(\time{next})) =
  \sum_{i =1}^N \grad_i V(x(\time{next})) \left( A_i x_i(\time{next}) +
    B_i \uself_i(\time{last}) \right) = 0 .
\end{align}
Unfortunately,~\eqref{eq:eventnetwork} requires global information and
cannot be checked in a distributed way. Instead, one can define a
local event that defines when a single agent $i \in \until{N}$ should
update its information as any time that
\begin{align}\label{eq:eventlocal}
  \grad_i V(x(t)) \left( A_i x_i(t) + B_i
    \uself_i(t) \right) = 0 .
\end{align}
As long as each agent~$i$ can ensure the local
event~\eqref{eq:eventlocal} has not yet occurred, it is guaranteed
that~\eqref{eq:eventnetwork} has not yet occurred either. The problem
with this approach is that each agent $i \in \until{N}$ needs to have
continuous access to information about the state of its neighbors
$\NN(i)$ in order to evaluate $ \grad_i V(x) = \grad_i
V(\availablestates^{i}) $ and check condition~\eqref{eq:eventlocal}.
The self-triggered approach removes this requirement on continuous
availability of information by having each agent employ instead the
possibly inexact information about the state of their neighbors.  The
notion of reachability set plays a key role in achieving this.  Given
$y \in \XX_i$, the \emph{reachable set} of points
under~\eqref{eq:agentdynamics} starting from $y$ in $s$ seconds is,
\begin{align*}
  \reachable_i(s,y) = \{ z \in \XX_i \; | \; \exists \, \map{u_i}{ [0,
    s] }{\UU_i} \text{\rm{ such that }} z = e^{A_i s} y + \int_{0}^{s}
  e^{A_i (s-\tau)} B_i u_i(\tau) d\tau \} .
\end{align*}
Using this notion, if agents have exact knowledge about the dynamics
and control sets of its neighboring agents (but not their
controllers), each agent can construct, each time state information is
received, sets that are guaranteed to contain their neighbors' states.

\begin{definition}[Guaranteed sets]
  {\rm If $\time{last}^i$ is the time at which agent~$i$ receives
    state information $x_j(\time{last}^i)$ from its neighbor $j \in
    \NN(i)$, then the \emph{guaranteed set} is given by
    \begin{align}\label{eq:guaranteed}
      \gset_j^i(t,\time{last}^i,x_j(\time{last}^i)) =
      \reachable_j(t-\time{last}^i,x_j(\time{last}^i)) \subset \XX_j ,
    \end{align}
    and is guaranteed to contain $x_j(t)$ for $t \geq
    \time{last}^i$.}
\end{definition}

We let $\gset_j^i(t) = \gset_j^i(t,\time{last}^i,x_j(\time{last}^i))$
when the starting state $x_j(\time{last}^i)$ and time $\time{last}^i$
do not need to be emphasized. We denote by $\guaranteedsets^i(t) = (
x_i(t), \{ \gset_j^i(t) \}_{j \in \NN(i)} )$ the information available
to agent $i$ at time $t$.

\begin{remark}[Computing reachable
  sets]\label{re:computing-reachable-sets} {\rm Finding the guaranteed
    or reachable sets~\eqref{eq:guaranteed} can be in general
    computationally expensive. A common approach consists of computing
    over-approximations to the actual reachable set via convex
    polytopes or ellipsoids. There exist efficient algorithms to
    calculate and store these for various classes of systems, see
    e.g.,~\cite{MA-CLG-BHK:11,GF:05}. Furthermore, agents can deal
    with situations where they do not have exact knowledge about the
    dynamics of their neighbors (so that the guaranteed sets cannot be
    computed exactly) by employing over-approximations of the actual
    guaranteed sets. \oprocend}
\end{remark}



With the guaranteed sets in place, we can now provide a test that
allows agents to determine when they should update their current
information and control signals.  
At time $\time{last}^i$, agent $i$ computes the next time
$\time{next}^i \ge \time{last}^i$ to acquire information via
\begin{align}\label{eq:selftest}
  \sup_{y_\NN \in \guaranteedsets^i(\time{next}^i)} \grad_i V(y_\NN)
  \left( A_i x_i(\time{next}^i) + B_i \uself_i(\time{next}^i) \right)
  = 0.
\end{align}
By~\eqref{eq:localassumption} and the fact that
$\gset_j^i(\time{last}^i) = \{ x_j(\time{last}^i) \}$, at time
$\time{last}^i$,
\begin{align*}
  \sup_{y_\NN \in \guaranteedsets^i(\time{last}^i)} \grad_i V(y_\NN)
  \left( A_i x_i(\time{last}^i) + B_i \uself_i(\time{last}^i) \right)
  = \grad_i V(\availablestates^i(\time{last}^i)) \left( A_i
    x_i(\time{last}^i) + B_i \uself_i(\time{last}^i) \right) \leq 0 .
\end{align*}
If all agents use this triggering criterion for updating information,
it is guaranteed that $ \frac{d}{dt} V (x(t))\le 0$ at all times
because, for each $i \in \until{N}$, the true state $x_j(t)$ is
guaranteed to be in $\gset_j^i(t)$ for all $j \in \NN(i)$ and $t \geq
\time{last}^i$.

The condition~\eqref{eq:selftest} is appealing because it can be
evaluated by agent $i$ with the information it possesses at time
$\time{last}^i$. Once determined, agent $i$ schedules that, at time
$\time{next}^i$, it will request updated information from its
neighbors. 
We refer to $\time{next}^i - \time{last}^i$ as the
\emph{self-triggered request time} for agent $i$. Due to the
conservative way in which $\time{next}^i$ is determined, it is
possible that $\time{next}^i = \time{last}^i$ for some $i$, which
would mean that instantaneous information updates are necessary (note
that this cannot happen for all $i \in \until{N}$ unless the network
state is already in $D$). This can be dealt with by introducing a
dwell time such that a minimum amount of time must pass before an
agent can request new information and using the safe-mode controller
while waiting for the new information.  We do not enter into details
here and defer the discussion to Section~\ref{se:selftime}.

The problem with the self-triggered approach is that the resulting
times are often conservative because the guaranteed sets can grow
large quickly as they capture all possible trajectories of neighboring
agents. It is conceivable that improvements can be made from tuning
the guaranteed sets based on what neighboring agents \emph{plan} to do
rather than what they \emph{can} do.  This observation is at the core
of the team-triggered approach proposed next.

\myclearpage
\section{Team-triggered coordination}\label{se:team-triggered}

This section presents the team-triggered approach for the real-time
implementation of distributed controllers on networked cyber-physical
systems. The team-triggered approach incorporates the reactive nature
of event-triggered approaches and, at the same time, endows individual
agents with the autonomy characteristic of self-triggered approaches
to determine when and what information is needed.
Agents make promises to their neighbors about their future
states and inform them if these promises are violated later (hence the
connection with event-triggered control). With the extra information
provided by the availability of the promises, each agent computes the
next time that an update is required and requests information from their
neighbors accordingly to guarantee the monotonicity of
the Lyapunov function $V$ introduced in Section~\ref{se:statement}
(hence the connection with self-triggered control).


\subsection{Promises}\label{se:promises}

A \emph{promise} can be either a time-varying set of states (state
promise) or controls (control promise) that an agent sends to another
agent. 

\begin{definition}[State promises and rules]
  {\rm A \emph{state promise} that agent $j\in \until{N}$ makes to
    agent $i$ at time $t$ is a set-valued, continuous (with respect to
    the Hausdorff distance) function $X_j^i[t] \in \mathcal{C}^0([t,
    \infty) ; \powerset(\XX_j))$.  A \emph{state promise rule} for
    agent $j \in \until{N}$ generated at time $t$ is a continuous
    (with respect to the distance $\functiondist$ defined
    in~\eqref{eq:dist-func}) map of the form $\staterule_j :
    \mathcal{C}^0 \left( [t,\infty); \prod_{i \in \NN(j) \cup \{ j\}}
      \powerset(\XX_i) \right) \rightarrow \mathcal{C}^0 \left(
      [t,\infty) ; \powerset \left( \XX_j \right) \right)$.}
\end{definition}

The notation $X_j^i[t](t')$ conveys the promise $x_j(t') \in
X_j^i[t](t')$ that agent $j$ makes at time $t$ to agent $i$ about
time~$t' \ge t$. A state promise rule is simply a way of generating
state promises.  This means that if agent $j$ must send information to
agent $i$ at time $t$, it sends the state promise $X_j^i[t] =
\staterule_j(\promisesets^j[\cdot]_{|[t,\infty)} )$. We require that,
in the absence of communication delays or noise in the state
measurements, the promises generated by a rule have the property that
$X_j^i[t](t) = \{ x_j(t) \}$.  For simplicity, when the time at which
a promise is received is not relevant, we use the notation
$X_j^i[\cdot]$, or simply $X_j^i$. All promise information available
to agent $i \in \until{N}$ at some time $t$ is given by
$\promisesets^i[\cdot]_{|[t, \infty)} = ({x_i}_{|[t,\infty)} , \{
X_j^i[\cdot]_{|[t,\infty)} \}_{j \in \NN(i)}) \in \mathcal{C}^0 \left(
  [t,\infty); \prod_{j \in \NN(i) \cup \{ i\}} \powerset(\XX_j)
\right)$. To extract information from this about a specific time $t'$,
we use $\promisesets^i[\cdot](t')$ or simply $\promisesets^i(t') = (
x_i(t') , \{ X_j^i[\cdot](t') \}_{j \in \NN(i)} ) \in \prod_{j \in
  \NN(i) \cup \{ i\}} \powerset(\XX_j)$. The generality of the above
definitions allow promise sets to be arbitrarily complex but we
restrict ourselves to promise sets that can be described with a finite
number of parameters.




\begin{remark}\longthmtitle{Example promise and
    rule}\label{re:control-promises-rules} 
  {\rm Alternative to directly sending state promises, agents can
    share their promises based on their control rather than their
    state. The notation $U_j^i[t](t')$ conveys the promise $u_j(t')
    \in U_j^i[t](t')$ that agent~$j$ makes at time $t$ to agent $i$
    about time~$t' \ge t$.  Given the dynamics of agent $j$ and state
    $x_j(t)$ at time $t$, agent $i$ can compute the state promise
    for~$t' \ge t$,
    \begin{align}\label{eq:promise}
      X_j^i[t](t') &= \{ z \in \XX_j \; | \; \exists \, \map{u_j}{ [t,
        t'] }{\UU_j} \text{ with } u_j(s) \in U_j^i[t](s) \text{ for }
      s \in [t,t']
      \\
      & \hspace*{3cm} \text{ such that } z = e^{A_j (t'-t)} x_j(t) +
      \int_{t}^{t'} e^{A_j (t'-\tau)} B_j u_j(\tau) d\tau \} . \notag
    \end{align}
    As an example, given $j \in \until{N}$, a continuous control law
    $u_j : \prod_{i \in \NN(j) \cup \{j \}} \powerset ( \XX_i )
    \rightarrow \UU_j$, and $\delta_j > 0$, the \emph{ball-radius
      control promise rule} for agent~$j$ generated at time~$t$ is
    \begin{align}\label{eq:staticpromise}
      \staticrule_j(\promisesets^j[\cdot]_{|[t,\infty)}) (t')=
      \cball{u_j(\promisesets^j(t))}{\delta_j} \cap \UU_j \qquad t'
      \geq t .
    \end{align}
    Note that this promise is a ball of radius $\delta_j$ in the
    control space $\UU_j$ centered at the control signal used at
    time~$t$.  Depending on whether $\delta_j$ is constant or changes
    with time, we refer to it as the static or dynamic ball-radius
    rule, respectively.  The promise can be sent with three
    parameters, the state $x_j(t)$ when the promise was sent, the
    control action $u_j(\promisesets^j(t))$ at that time, and the
    radius~$\delta_j$ of the ball. The state promise can then
    be generated using~\eqref{eq:promise}. \oprocend
  }
\end{remark}



Promises allow agents to predict the evolution of their neighbors more
accurately, which directly affects the network behavior.  In general,
tight promises correspond to agents having good information about
their neighbors, which at the same time may result in an increased
communication effort (since the promises cannot be kept for long
periods of time).  On the other hand, loose promises correspond to
agents having to use more conservative controls due to the lack of
information, while at the same time potentially being able to operate
for longer periods of time without communicating (because promises are
not violated).

The availability of promises equips agents with set-valued information
models about the state of other agents. This fact makes it necessary
to address the definition of distributed controllers that operate on
sets, rather than points. We discuss this in
Section~\ref{se:controllers-set-valued}.  The additional information
that promises represent is beneficial to the agents because it
decreases the amount of uncertainty when making action
plans. Section~\ref{se:selftime} discusses this in detail.  Finally,
these advantages rely on the assumption that promises hold throughout
the evolution. As the state of the network changes and the level of
task completion evolves, agents might decide to break former promises
and make new ones. We examine this in Section~\ref{se:brokenpromises}.


\subsection{Controllers on set-valued information
  models}\label{se:controllers-set-valued}

Here we discuss the type of controllers that the team-triggered
approach relies on.  The underlying idea is that, since agents possess
set-valued information about the state of other agents through
promises, controllers themselves should be defined on sets, rather
than on points.  There are different ways of designing controllers
that operate with set-valued information depending on the type of
system, its dynamics, or the desired task, see
e.g.,~\cite{FB-SM:08}. For the problem of interest here, we offer the
following possible goals. One may be interested in simply decreasing
the value of a Lyapunov function as fast as possible, at the cost of
more communication or sensing. Alternatively, one may be interested in
choosing the stabilizing controller such that the amount of required
information is minimal at a cost of slower convergence time.
We consider continuous (with respect to the Hausdorff distance)
controllers of the form $u^{**} : \prod_{j \in \until{N}} \powerset(
\XX_j ) \rightarrow \real^{m}$ that satisfy
\begin{subequations}\label{eq:social-assumptions}
  \begin{align}
    \grad_i V(x) \left( A_i x_i + B_i u^{**}_i( \{ x \}) \right) &
    \leq 0, \quad i \in \until{N},
    \label{eq:localassumptionsocial}
    \\
    \sum_{i = 1}^N \grad_i V(x) \left( A_i x_i + B_i u^{**}_i( \{ x
      \}) \right) & < 0.
    \label{eq:desiredderivativesocial}
  \end{align}
\end{subequations}
In other words, if exact, singleton-valued information is available to
the agents, then the controller~$u^{**}$ guarantees the monotonic
evolution of the Lyapunov function~$V$. We assume that $u^{**}$ is
distributed over the communication graph~$\commgraph$. As before, this
means that for each $i \in \until{N}$, the $i$th component $u^{**}_i$
can be computed with information in $\prod_{j \in \NN(i) \cup \{i\}}
\powerset( \XX_j )$ rather than in the full space $\prod_{j \in
  \until{N}} \powerset( \XX_j )$.

Controllers of the above form can be derived from the availability of
the controller $\map{u^*}{\XX}{\UU}$ introduced in
Section~\ref{se:statement}.  Specifically, let $E: \prod_{j=1}^N
\powerset ( \XX_j ) \rightarrow \XX$ be a continuous map that is
distributed over~$\commgraph$ and satisfies, for each $i \in
\until{N}$, that $E_i(Y) \in Y_i$ for each $Y \in \prod_{j=1}^N
\powerset ( \XX_j )$ and $E_i( \{ y \} ) = y_i$ for each $y \in \XX$.
Essentially, what the map $E$ does for each agent is select a point
from the set-valued information that it possesses.  Now, define
\begin{align}\label{eq:examplecontrol}
  u^{**}(Y) = u^*( E(Y)).
\end{align}
Note that this controller satisfies~\eqref{eq:localassumptionsocial}
and~\eqref{eq:desiredderivativesocial} because~$u^*$
satisfies~\eqref{eq:localassumption} and~\eqref{eq:desiredderivative}.

\begin{example}[Controller definition with the ball-radius promise
  rules]\label{ex:estimate}
  {\rm Here we construct a controller $u^{**}$
    using~\eqref{eq:examplecontrol} for the case when promises are
    generated according to the ball-radius control rule described in
    Remark~\ref{re:control-promises-rules}.  To do so, note that it is
    sufficient to define the map $E: \prod_{j=1}^N \powerset ( \XX_j )
    \rightarrow \XX$ only for tuples of sets of the form given
    in~\eqref{eq:promise}, where the corresponding control promise is
    defined by~\eqref{eq:staticpromise}. With the notation of
    Remark~\ref{re:control-promises-rules}, recall that the promise
    that an agent $j$ sends at time $t$ is conveyed through three
    parameters $(y_j,v_j,\delta_j)$, the state $y_j = x_j(t)$ when the
    promise was sent, the control action $v_j =
    u_j(\promisesets^j(t))$ at that time, and the radius~$\delta_j$ of
    the ball. We can then define the $j$th component of the map $E$ as
    \begin{align*}
      E_j(X_1[t](t'),\dots,X_N[t](t')) = e^{A_j (t'-t)} y_j +
      \int_{t}^{t'} e^{A_j (t'-\tau)} B_j v_j d\tau ,
    \end{align*}
    which is guaranteed to be in $X_j[t](t')$ for~$t' \geq t$. This
    specification amounts to each agent $i$ calculating the evolution
    of its neighbors $j \in \NN(i)$ as if they were using a zero-order
    hold control.
    \oprocend}
\end{example}

\subsection{Self-triggered information updates}\label{se:selftime}

Here we discuss how agents use the promises received from other agents
to generate self-triggered information requests in the future.  Let
$\time{last}^i$ be some time at which agent $i$ receives updated
information (i.e., promises) from its neighbors.  Until the next time
information is obtained, agent~$i$ has access to the collection of
functions $\promisesets^i$ describing its neighbors' state and can
compute its evolution under the controller~$u^{**}$ via
\begin{align}\label{eq:compute-own-evolution}
  x_i(t) = e^{A_i (t-\time{last}^i)} x_i(\time{last}^i) +
  \int_{\time{last}^i}^{t} e^{A_i (t-\tau)} B_i
  u^{**}_i(\promisesets^i(\tau)) d\tau , \quad t \ge \time{last}^i .
\end{align}
Note that this evolution of agent $i$ can be viewed as a promise that
it makes to itself, i.e., $X_i^i[\cdot](t) = \{ x_i(t) \}$.  With this
in place, $i$ can schedule the next time $\time{next}^i$ at which it
will need updated information from its neighbors by computing the
worst-case time evolution of $V$ along its trajectory among all the
possible evolutions of its neighbors given the information contained
in their promises. Formally,
we define, for $Y_\NN \in \prod_{j \in \NN(i) \cup \{ i \}} \powerset
( \XX_j )$,
\begin{align}\label{eq:upper-bound-Lie-derivative}
  \si(Y_\NN) = \sup_{y_\NN \in Y_\NN} \grad_i V(y_\NN) \left( A_i y_i
    + B_i u_i^{**}( Y_\NN ) \right) ,
\end{align}
where $y_i$ is the element of $y_\NN$ corresponding to $i$.  Then, the
trigger for when agent $i$ needs new information from its neighbors is
similar to~\eqref{eq:selftest}, where we now use the promise sets
instead of the guaranteed sets. Specifically, the critical time at
which information is requested is given by $\time{next}^i = \max \{
\time{last}^i + \dwellself, t^* \}$, where $\dwellself > 0$ is an a
priori chosen parameter that we discuss below and $t^*$ is implicitly
defined by
\begin{align}\label{eq:socialselftrigger}
t^* = \min \setdef{t \geq \time{last}^i}{\si(\promisesets^i(t)) = 0}.
\end{align}
This ensures that for $t \in [\time{last}^i, t^*)$, agent $i$ is
guaranteed to be contributing positively to the desired task.
We refer to $\time{next}^i - \time{last}^i$ as the self-triggered
request time.
The parameter $\dwellself > 0$ is the \emph{self-triggered dwell
  time}. We introduce it because, in general, it is possible that $t^*
= \time{last}^i$, implying that instantaneous communication is
required.  The dwell time is used to prevent this behavior as
follows. Note that $\si(\promisesets^i(t')) \leq 0$ is only guaranteed
while $t'\in [\time{last}^i, t^*]$. Therefore, in case that
$\time{next}^i=\time{last}^i + \dwellself$, i.e., if $t^* \leq
\time{last}^i + \dwellself$, agent $i$ uses the safe-mode control
during $t' \in (t^*, \time{last}^i+\dwellself]$ to leave its state
fixed. This design ensures the monotonicity of the evolution of~$V$
along the network execution.  The team-triggered controller is defined~by
\begin{align}\label{eq:agentsocialcontrol}
  \uteam_i(t) = \begin{cases} u_i^{**}(\promisesets^i(t)) , & \text{if
    } t \leq t^* ,
    \\
    \safecontrol_i(x_i(t)) , & \text{if } t > t^*
    ,
  \end{cases}
\end{align}
for $t \in [\time{last}^i,\time{next}^i)$, where $t^*$ is given
by~\eqref{eq:socialselftrigger}.  Note that the self-triggered dwell
time~$\dwellself$ only limits the frequency at which an agent~$i$ can
\emph{request} information from its neighbors and does not provide
guarantees on inter-event times of when its memory is updated or its
control is recomputed. If a neighboring agent sends information to
agent $i$ before this dwell time has expired (because that agent has
broken a promise), this triggers agent $i$ to update its memory and
potentially recompute its control law.




\subsection{Event-triggered information
  updates}\label{se:brokenpromises}

Agent promises may need to be broken for a variety of reasons. For
instance, an agent might receive new information from its neighbors,
causing it to change its former plans. Another example is given by an
agent that made a promise that is not able to keep for as long as it
anticipated.
Consider an agent $i \in \until{N}$ that has sent a promise
$X_i^j[\time{last}]$ to a neighboring agent $j$ at some
time~$\time{last}$.  If agent $i$ ends up breaking its promise at time
$t^* \geq \time{last}$, i.e., $x_i(t^*) \notin
X_i^j[\time{last}](t^*)$, then it is responsible for sending a new
promise~$X_i^j[\time{next}]$ to agent~$j$ at time $\time{next} = \max
\{ \time{last} + \dwellevent, t^* \}$, where $\dwellevent > 0$ is an a
priori chosen parameter that we discuss below.  This implies that~$i$
must keep track of promises made to its neighbors and monitor them in
case they are broken. Note that this mechanism is implementable
because each agent only needs information about its own state and the
promises it has made to determine whether the trigger is satisfied.



The parameter $\dwellevent > 0$ is known as the \emph{event-triggered
  dwell time}. We introduce it because, in general, the time $t^* -
\time{last}$ between when agent $i$ makes and breaks a promise to an
agent~$j$ might be arbitrarily small. The issue, however, is that if
$t^* < \time{last} + \dwellevent$, agent $j$ operates under incorrect
information about agent $i$ for $t \in [t^*, \time{last} +
\dwellevent)$. We deal with this by introducing a warning message WARN
that agent $i$ must send to agent $j$ when it breaks its promise at
time $t^* < \time{last} + \dwellevent$.
If agent $j$ receives such a warning message, it redefines the promise
$X_i^j$ using the guaranteed sets~\eqref{eq:guaranteed} as follows,
\begin{align}\label{eq:promise-delay}
  X_i^j[\cdot](t) &= \bigcup_{x_i \in X_i^j[\cdot](t^*)}
  \gset_i^j(t,x_i) = \bigcup_{x_i \in X_i^j[\cdot](t^*)}
  \reachable_i(t-t^*,x_i)
\end{align}
for $t \geq t^*$, until the new message arrives at time $\time{next} =
\time{last} + \dwellevent$. By definition of the reachable set, the
promise $X_i^j[\cdot](t)$ is guaranteed to contain $x_i(t)$ for $t
\geq t^*$.

\begin{remark}[Promise expiration times]\label{re:expiration} {\rm It
    is also possible to set an expiration time~$\expirationevent >
    \dwellevent$ for the validity of promises. If this in effect and a
    promise is made at $\time{last}$, it is only valid for $t \in
    [\time{last}, \time{last} + \expirationevent]$. The expiration of
    the promise triggers the formulation of a new one.  \oprocend }
\end{remark}

The combination of the self- and event-triggered information updates
described above together with the team-triggered controller~$\uteam$
as defined in~\eqref{eq:agentsocialcontrol} gives rise to the
\algosocial, which is formally presented in
Algorithm~\ref{tab:social}. The self-triggered information request in
Algorithm~\ref{tab:social} is executed by an agent anytime new
information is received, whether it was actively requested by the
agent, or was received from some neighbor due to the breaking of a
promise. 

\begin{algorithm}[htb]
  {\footnotesize \vspace*{1ex}
    \emph{(Self-trigger information update)}
    \\
    \parbox{\linewidth}{At any time $t$ agent $i\in \until{N}$
    receives new promise(s) $X_j^i[t]$ from neighbor(s) $j \in
    \NN(i)$, agent $i$ performs:}\\[-4ex]
    \begin{algorithmic}[1]
      %
      \STATE compute own control $\uteam_i(t')$ for $t' \geq t$
      using~\eqref{eq:agentsocialcontrol}
      \STATE compute own state evolution $x_i(t')$ for $t'\ge t$
      using~\eqref{eq:compute-own-evolution}
      \STATE compute first time $t^* \geq t$ such that $
      \si(\promisesets^i(t^*)) = 0 $
      \STATE schedule information request to neighbors in $\max \{ t^*
      - t, \dwellself \}$ seconds
    \end{algorithmic}
    \emph{(Respond to  information request)}
    \\
    At any time $t$ a neighbor $j \in \NN(i)$ requests information,
    agent $i$ performs:
    \begin{algorithmic}[1]
      \STATE send new promise $X_i^j[t] =
      \staterule_i(\promisesets^i[\cdot]_{[t,\infty)})$ to agent $j$ 
    \end{algorithmic}
    \emph{(Event-trigger information update)}
    \\
    At all times $t$, agent $i$ performs:
    \begin{algorithmic}[1]
      \IF {there exists $j \in \NN(i)$ such that $x_i(t) \notin
        X_i^j[\cdot](t)$}
      %
      \IF {agent $i$ has sent a promise to $j$ at some time $\time{last} \in (t - \dwellevent,
      t]$}
      \STATE send warning message WARN to agent $j$ at time $t$
      \STATE schedule to send new promise $X_i^j[\time{last}+\dwellevent] =
      \staterule_i(\promisesets^i[\cdot]_{|[\time{last}+\dwellevent,\infty)})$ to agent $j$ in $\time{last} + \dwellevent - t$ seconds
      \ELSE
      \STATE send new promise $X_i^j[t] =
      \staterule_i(\promisesets^i[\cdot]_{|[t,\infty)})$ to agent $j$ at time $t$
      \ENDIF
      \ENDIF
    \end{algorithmic}
    \emph{(Respond to warning message)}
    \\
    \parbox{\linewidth}{ At any time $t$ agent $i \in \until{N}$
      receives a warning message WARN from agent $j \in \NN(i)$ }
    \\[-4ex]
    \begin{algorithmic}[1]
      \STATE redefine promise set $X_j^i[\cdot](t') = \cup_{x_j \in
        X_j^i[\cdot](t)} \reachable_j(t'-t,x_j)$ for $t' \geq t$
    \end{algorithmic}}
  \caption{\hspace*{-.5ex}: \algosocial }\label{tab:social}
\end{algorithm}


\myclearpage
\section{Convergence of the \algosocial}\label{se:analysis}

Here we analyze the convergence properties of the \algosocial.  Our
first result establishes the monotonic evolution of the Lyapunov
function~$V$ along the network trajectories.

\begin{proposition}\label{prop:monotonicfull}
  Consider a networked cyber-physical system as described in
  Section~\ref{se:statement} executing the \algosocial
  (cf. Algorithm~\ref{tab:social}) based on a continuous controller
  $u^{**} : \prod_{j \in \until{N}} \powerset( \XX_j ) \rightarrow
  \real^{m}$ that satisfies~\eqref{eq:social-assumptions} and is
  distributed over the communication graph~$\commgraph$.  Then, the
  function $V$ is monotonically nonincreasing along any network
  trajectory.
\end{proposition}
\begin{proof}
  We start by noting that the time evolution of $V$ under
  Algorithm~\ref{tab:social} is continuous and piecewise continuously
  differentiable. Moreover, at the time instants when the time
  derivative is well-defined, one has
  \begin{align}\label{eq:socialeventlocal}
    \frac{d}{dt} V(x(t)) & = \sum_{i=1}^N \grad_i
    V(\availablestates^i(t)) \left( A_i x_i(t) + B_i \uteam_i(t)
    \right)
    \\
    & \leq \sum_{i=1}^N \sup_{y_\NN \in \promisesets^i(t)} \grad_i
    V(y_\NN) \left( A_i x_i(t) + B_i \uteam_i(t) \right) \leq
    0. \notag
  \end{align}
  As we justify next, the last inequality follows by design of the
  \algosocial. For each $i \in \until{N}$, if $\si(\promisesets^i(t))
  \leq 0$, then $\uteam_i(t) = u_i^{**}(\promisesets^i(t))$
  (cf.~\eqref{eq:agentsocialcontrol}). In this case the corresponding
  summand of~\eqref{eq:socialeventlocal} is
  exactly~$\si(\promisesets^i(t))$, as defined
  in~\eqref{eq:upper-bound-Lie-derivative}.  If
  $\si(\promisesets^i(t)) > 0$, then $\uteam_i(t) =
  \safecontrol_i(x_i(t))$, for which the corresponding summand
  of~\eqref{eq:socialeventlocal} is exactly $0$.
\end{proof}

The next result characterizes the convergence properties of
team-triggered coordination strategies.

\begin{proposition}\label{prop:main-convergence}
  Consider a networked cyber-physical system as described in
  Section~\ref{se:statement} executing the \algosocial
  (cf. Algorithm~\ref{tab:social}) with dwell times $\dwellself,
  \dwellevent > 0$ based on a continuous controller $u^{**} : \prod_{j
    \in \until{N}} \powerset( \XX_j ) \rightarrow \real^{m}$ that
  satisfies~\eqref{eq:social-assumptions} and is distributed over the
  communication graph~$\commgraph$.  Then, any bounded network
  trajectory with uniformly bounded promises asymptotically approaches
  the desired set~$D$.
\end{proposition}

The requirements of uniformly bounded promises in
Proposition~\ref{prop:main-convergence} means that there exists a
compact set that contains all promise sets.
Note that this is automatically guaranteed if the network state space
is compact.
Alternatively, if the sets of allowable controls are bounded, a
bounded network trajectory with expiration times for promises
implemented as outlined in Remark~\ref{re:expiration} would result in
uniformly bounded promises.  There are two main challenges in proving
Proposition~\ref{prop:main-convergence}, which we discuss next.

The first challenge is that agents operate asynchronously, i.e.,
agents receive and send information, and update their control laws
possibly at different times.
To model asynchronism, we use a procedure called analytic
synchronization, see e.g.~\cite{JL-ASM-BDOA:07b}.  Let the time
schedule of agent $i$ be given by $\timeschedule^i = \{ t_0^i, t_1^i,
\dots \}$, where $t_\ell^i$ corresponds to the $\ell$th time that
agent $i$ receives information from one or more of its neighbors (the
time schedule $\timeschedule^i$ is not known a priori by the agent).
Note that this information can be received because $i$ requests it
itself, or a neighbor sends it to $i$ because an event is triggered.
Analytic synchronization simply consists of merging together the
individual time schedules into a global time schedule $\timeschedule =
\{t_0, t_1, \dots \}$ by setting
\begin{align*}
  \timeschedule = \cup_{i=1}^N \timeschedule^i .
\end{align*}
Note that more than one agent may receive information at any given
time $t \in \timeschedule$.  This synchronization is done for analysis
purposes only. For convenience, we identify $\integernonnegative$ with
$\timeschedule$ via~$\ell \mapsto t_\ell$.

The second challenge is that a strategy resulting from the
team-triggered approach has a discontinuous dependence on the network
state and the agent promises. More precisely, the information
possessed by any given agent are trajectories of sets for each of
their neighbors, i.e., promises. For convenience, we denote by
\begin{align*}
  S & = \prod_{i=1}^N S_i, \quad \text{where}
  \\
  S_i & = \mathcal{C}^0 \Big( \real ; \powerset (\XX_1) \times \dots
  \times \powerset (\XX_{i-1}) \times \XX_i \times \powerset
  (\XX_{i+1}) \times \dots \times \powerset (\XX_{N})
  \Big) ,
\end{align*}
the space that the state of the entire network lives in. Note that
this set allows us to capture the fact that each agent $i$ has perfect
information about itself, as described in Section~\ref{se:statement}.
Although agents only have information about their neighbors, the above
space considers agents having promise information about all other
agents to facilitate the analysis.  This is only done to allow for a
simpler technical presentation, and does not impact the validity of
the arguments made here.  The information possessed by all agents of
the network at some time $t$ is collected in
\begin{align*}
  \left( X^1[\cdot]_{|[t,\infty)}, \dots, X^N[\cdot]_{|[t,\infty)}
  \right) \in S,
\end{align*}
where $ X^i[\cdot]_{|[t,\infty)} = \left( X_1^i[\cdot]_{|[t,\infty)},
  \dots, X_N^i[\cdot]_{|[t,\infty)} \right) \in S_i$.  Here, $[\cdot]$
is shorthand notation to denote the fact that promises might have been
made at different times, earlier than $t$.  The \algosocial
corresponds to a discontinuous map of the form $S \times
\integernonnegative \rightarrow S \times \integernonnegative$. This
fact makes it difficult to use standard stability methods to analyze
the convergence properties of the network. Our approach to this
problem consists of defining a discrete-time set-valued map $\algomap
: S \times \integernonnegative \rightrightarrows S \times
\integernonnegative$ whose trajectories contain the trajectories of
the \algosocial. Although this `over-approximation procedure' enlarges
the set of trajectories to consider, the gained benefit is that of
having a set-valued map with suitable continuity properties that is
amenable to set-valued stability analysis. We describe this in detail
next.

We start by defining the set-valued map~$M$.  Let $(Z, \ell) \in S
\times \integernonnegative$.  We define the $(N+1)$th component of all
the elements in $\algomap(Z,\ell)$ to be $\ell+1$.  The $i$th
component of the elements in $\algomap(Z,\ell)$ is given by one of
following possibilities.  The first possibility models the case when
agent $i$ does not receive any information from its neighbors. In this
case, the $i$th component of the elements in $\algomap(Z,\ell)$ is
simply the $i$th component of $Z$,
\begin{align}\label{eq:noupdate}
  \left( {Z_1^i}_{|[t_{\ell+1},\infty)}, \dots,
    {Z_N^i}_{|[t_{\ell+1},\infty)} \right),
\end{align}
The second possibility models the case when agent $i$ has received
information (including a WARN message) from at least one neighbor: the
$i$th component of the elements in $\algomap(Z,\ell)$ is
\begin{align}\label{eq:update}
  \left( {Y_1^i}_{|[t_{\ell+1},\infty)}, \dots,
    {Y_N^i}_{|[t_{\ell+1},\infty)} \right),
\end{align}
where each agent has access to its own state at all times,
\begin{subequations}
  \begin{align}\label{eq:Y^i_i}
    Y_i^i(t) = e^{A_i (t-t_{\ell+1})} Z_i^i(t_{\ell+1}) +
    \int_{t_{\ell+1}}^t e^{A_i (t-\tau)} B_i \uteam_i(\tau) d\tau ,
    \quad t \geq t_{\ell+1} ,
  \end{align}  
  (here, with a slight abuse of notation, we use $\uteam$ to denote
  the controller evaluated at $Y^i$) and,
  \begin{align}
    {Y_j^i}_{|[t_{\ell+1},\infty)} =
    \begin{cases}
      {Z_j^i}_{|[t_{\ell+1},\infty)} , & \text{if } i \text{ does
        not receive information from } j ,
      \\
      {W_j^i}_{|[t_{\ell+1},\infty)} , & \text{if } i \text{
        receives a warning message from } j ,
      \\
      \staterule_j({Z^j_\NN}_{|[t_{\ell+1},\infty)}) , &
      \text{otherwise} ,
    \end{cases}
  \end{align}
  for $j \neq i$, where $W_j^i(t) = \bigcup_{z_i \in
    Z^i_j(t_{\ell+1})} \gset_j^i(t,z_i)$ corresponds to the redefined
  promise~\eqref{eq:promise-delay} for $t \geq t_{\ell+1}$ as a result
  of the warning message.
\end{subequations}



We emphasize two properties of the set-valued map $M$. First, any
trajectory of the \algosocial is also a trajectory of the
non-deterministic dynamical system defined by~$\algomap$,
\begin{align*}
  (Z(t_{\ell+1}), \ell+1) \in \algomap(Z(t_\ell), \ell) .
\end{align*}
Second, although the map defined by the \algosocial is
discontinuous, the set-valued map $M$ is closed, as we show next (a
set-valued map $T: X \rightrightarrows Y$ is closed if $x_k
\rightarrow x$, $y_k \rightarrow y$ and $y_k \in T(x_k)$ imply that $y
\in T(x)$).

\begin{lemma}[Set-valued map is closed]\label{le:M-closed}
  The set-valued map~$\algomap : S \times \integernonnegative
  \rightrightarrows S \times \integernonnegative$ is closed.
\end{lemma}
\begin{proof}
  To show this we appeal to the fact that a set-valued map composed of
  a finite collection of continuous maps is
  closed~\cite[E1.9]{FB-JC-SM:08cor}.  Given $(Z, \ell)$, the set
  $\algomap(Z, \ell)$ is finitely comprised of all possible
  combinations of whether or not updates occur for every agent pair
  $i, j \in \until{N}$.  In the case that an agent $i$ does not
  receive any information from its neighbors, it is trivial to show
  that~\eqref{eq:noupdate} is continuous in $(Z, \ell)$ because
  ${Z^i_j}_{[t_{\ell+1},\infty)}$ is simply the restriction of
  ${Z^i_j}_{[t_\ell,\infty)}$ to the interval $[t_{\ell+1},\infty)$,
  for each $i \in \until{N}$ and $j \in \NN(i)$. In the case that an
  agent $i$ does receive updated information, the above argument still
  holds for agents $j$ that did not send information to agent~$i$.  If
  an agent $j$ sends a warning message to agent $i$,
  ${W_j^i}_{|[t_{\ell+1},\infty)}$ is continuous in $(Z, \ell)$ by
  continuity of the reachable sets on their starting point.  If an
  agent $j$ sends a new promise to agent $i$,
  ${Y_j^i}_{|[t_{\ell+1},\infty)}$ is continuous in $(Z, \ell)$ by
  definition of the function $\staterule_j$.  Finally, one can see
  that ${Y_i^i}_{|[t_{\ell+1},\infty)}$ is continuous in $(Z,\ell) $
  from~\eqref{eq:Y^i_i}.
\end{proof}

We are now ready to prove Proposition~\ref{prop:main-convergence}.

\begin{proof}[Proof of Proposition~\ref{prop:main-convergence}]
  Here we resort to the LaSalle Invariance Principle for set-valued
  discrete-time dynamical systems~\cite[Theorem
  1.21]{FB-JC-SM:08cor}. Let $ W = S \times \integernonnegative$,
  which is closed and strongly positively invariant with respect
  to~$\algomap$.  A similar argument to that in the proof of
  Proposition~\ref{prop:monotonicfull} shows that the function $V$ is
  nonincreasing along~$\algomap$.  Combining this with the fact that
  the set-valued map~$\algomap$ is closed
  (cf. Lemma~\ref{le:M-closed}), the application of the LaSalle
  Invariance Principle implies that the trajectories of~$\algomap$
  that are bounded in the first $N$ components approach the largest
  weakly positively invariant set contained in
  \begin{align}\label{eq:convergenceset}
    S^* &= \setdef{ (Z, \ell) \in S \times \integernonnegative}{
      \exists (Z', \ell+1) \in \algomap(Z, \ell) \text{ such that }
      V(Z') = V(Z) } , \notag
    \\
    &= \setdef{ (Z, \ell) \in S \times \integernonnegative}{
      \si(Z^i_\NN) \geq 0 \text{ for all } i \in \until{N} } .
  \end{align}

  We now restrict our attention to those trajectories of~$\algomap$
  that correspond to the \algosocial.  For convenience, let 
  $\loc(Z,\ell) : S \times \integernonnegative \rightarrow \XX$ be the
  map that extracts the true position information in $(Z, \ell)$,
  i.e.,
  \begin{align*}
    \loc(Z,\ell) = \left( Z_1^1(t_\ell), \dots, Z_N^N(t_\ell) \right)
    .
  \end{align*}
  Given a trajectory $\gamma$ of the \algosocial that satisfies all the
  assumptions of the statement of Proposition~\ref{prop:main-convergence}, 
  the bounded evolutions and uniformly bounded promises ensure that the
  trajectory~$\gamma$ is bounded. Then, the omega
  limit set $\Omega(\gamma)$ is weakly positively invariant and hence
  is contained in $ S^*$.
  Our objective is to show that, for any $(Z, \ell) \in
  \Omega(\gamma)$, we have $\loc(Z, \ell) \in D$. We show this
  reasoning by contradiction.  Let $(Z, \ell) \in \Omega(\gamma)$ but
  suppose $\loc(Z, \ell) \notin D$.  This means that $\si(Z^i_\NN)
  \geq 0$ for all $i \in \until{N}$. Take any agent $i$, by the
  \algoselfrequests, agent $i$ will request new information from
  neighbors in at most $\dwellself$ seconds. This means there exists a
  state $(Z', \ell+\ell') \in \Omega(\gamma)$ for which agent $i$ has
  just received updated information from its neighbors $j \in
  \NN(i)$. Since $(Z', \ell+\ell') \in S^*$, we know $\si({Z^i_\NN}')
  \geq 0$. We also know, since information was just updated, that
  ${Z_j^i}' = \loc_j(Z', \ell+\ell')$ is exact for all $j \in \NN(i)$.
  But, by~\eqref{eq:localassumptionsocial}, $\si({Z^i_\NN}') \leq 0$
  because $\loc(Z', \ell+\ell') \notin D$. This means that each time
  any agent $i$ updates its information, we must have $\si({Z^i_\NN}')
  = 0$. However, by~\eqref{eq:desiredderivativesocial}, there must
  exist at least one agent $i$ such that $\si({Z^i_\NN}') < 0$ since
  $\loc(Z',\ell+\ell') \notin D$, which yields a contradiction. Thus
  for the trajectories of the \algosocial, $(Z,\ell) \in S^*$ implies
  that $\loc(Z,\ell) \in D$.
\end{proof}

Given the convergence result of
Proposition~\ref{prop:main-convergence}, a termination condition for
the \algosocial could be included via the implementation of a
distributed algorithm that employs tokens identifying what agents are
using safe-model controllers, see e.g.,~\cite{NAL:97,DP:00}.  Also,
according to the proof of Proposition~\ref{prop:main-convergence}, the
actual value of the event-triggered dwell time~$\dwellevent$ does not
affect the convergence property of the trajectories of the constructed
discrete-time set-valued system.  However, the dwell time does affect
the rate of convergence of the actual continuous-time system (as a
larger dwell time corresponds to more time actually elapsing between
each step of the constructed discrete-time system).

\begin{remark}[Availability of a safe-mode
  controller]\label{rem:safe-mode-controller}
  {\rm The assumption on the availability of the safe-mode controller
    plays an important role in the proof of
    Proposition~\ref{prop:main-convergence} because it provides
    individual agents with a way of avoiding having a negative impact
    on the monotonic evolution of the Lyapunov function.  We believe
    this assumption can be relaxed for dynamics that allow agents to
    execute maneuvers that bring them back to their current
    state. Under such maneuvers, the Lyapunov function will not evolve
    monotonically but, at any given time, will always guarantee to
    be less than or equal to its current value at some future time. We
    have not pursued this approach here for simplicity and instead
    defer it for future work.  } \oprocend
 \end{remark}

 The next result states that, under the \algosocial with positive
 dwell times, the system does not exhibit Zeno behavior.
  
\begin{lemma}[No Zeno behavior]\label{le:zeno}
  Under the assumptions of Proposition~\ref{prop:main-convergence},
  the network executions do not exhibit Zeno behavior.
\end{lemma}
\begin{proof}
  Due to the self-triggered dwell time~$\dwellself$, the
  self-triggered information request steps in
  Algorithm~\ref{tab:social} guarantee that the minimum time before an
  agent $i$ asks its neighbors for new information is $\dwellself >
  0$. Similarly, due to the event-triggered dwell time~$\dwellevent$,
  agent $i$ will never receive more than two messages (one accounts
  for promise information, the other for the possibility of a WARN
  message) from a neighbor $j$ in a period of $\dwellevent > 0$
  seconds. This means that any given agent can never receive an
  infinite amount of information in finite time.
  When new information is received, the control
  law~\eqref{eq:agentsocialcontrol} can only switch a maximum of two
  times until new information is received again. Specifically, if an
  agent $i$ is using the normal control law when new information is
  received, it may switch to the safe-mode controller at most one time
  until new information is received again. If instead an agent $i$ is
  using the safe-mode control controller when new information is
  received, it may immediately switch to the normal control law, and
  then switch back to the safe-mode controller some time in the future
  before new information is received again.  The result follows from
  the fact that $|\NN(i)|$ is finite for each $i \in \until{N}$.
\end{proof}

\begin{remark}[Adaptive self-triggered dwell time]
  {\rm 
    Dwell times play an important role in preventing Zeno
    behavior. However, a constant self-triggered dwell time throughout
    the network evolution might result in wasted communication effort
    because some agents might reach a state where their effect on the
    evolution of the Lyapunov function is negligible compared to
    others. In such case, the former agents could implement larger
    dwell times, thus decreasing communication effort, without
    affecting the overall performance. Next, we give an example of
    such an adaptive dwell time scheme.
    Let $t$ be a time at which agent $i \in \until{N}$ has just
    received new information from its neighbors $\NN(i)$. Then, the
    agent sets its dwell time to
    \begin{align}\label{eq:adaptive-dwell}
      \dwellself^i(t) = \max \bigg\lbrace \delta_d \sum_{j \in \NN(i)}
        \frac{1}{| \NN(i) |} \frac{ \TwoNorm{
            u^{**}_j(\promisesets^j(t)) - \safecontrol_j(x_j(t)) ) }}{
          \TwoNorm{ u^{**}_i(\promisesets^i(t)) -
            \safecontrol_i(x_i(t)) }} , \Delta_d \bigg\rbrace,
    \end{align}
    for some a priori chosen $\delta_d$, $\Delta_d > 0$.  The
    intuition behind this design is the following.  The value
    $\TwoNorm{ u^{**}_j(\promisesets^j(t)) - \safecontrol_j(x_j(t)) )
    }$ can be interpreted as a measure of how far agent $j$ is from
    reaching a point where it cannot no longer contribute positively
    to the global task.  As agents are nearing this point, they are
    more inclined to use the safe mode control to stay put and hence
    do not require fresh information. Therefore, if agent $i$ is close
    to this point but its neighbors are not,~\eqref{eq:adaptive-dwell}
    sets a larger self-triggered dwell time to avoid excessive
    requests for information.  Conversely, if agent $i$ is far from
    this point but its neighbors are not,~\eqref{eq:adaptive-dwell}
    sets a small dwell time to let the self-triggered request
    mechanism be the driving factor in determining when new
    information is needed.  For agent $i$ to implement this, in
    addition to current state information and promises, each neighbor
    $j \in \NN(i)$ also needs to send the value of $\TwoNorm{
      u^{**}_j(\promisesets^j(t)) - \safecontrol_j(x_j(t)) ) }$ at
    time $t$.
    In the case that information is not received from all neighbors,
    agent $i$ simply uses the last computed dwell time.
    Section~\ref{se:simulations} illustrates this adaptive scheme in
    simulation. \oprocend}
\end{remark}

\section{Robustness against unreliable
  communication}\label{se:robustness}
  
This section studies the robustness of the team-triggered approach in
scenarios with packet drops, delays, and communication noise. We start
by introducing the possibility of packet drops in the network. For any
given message an agent sends to another agent, assume there is an
unknown probability $0 \le p < 1$ that the packet is dropped, and the
message is never received. We also consider an unknown (possibly
time-varying) communication delay $\delay(t) \leq \maxdelay$ in the
network for all $t$ where $\maxdelay \geq 0$ is known. In other words,
if agent $j$ sends agent $i$ a message at time $t$, agent $i$ will not
receive it with probability $p$ or receive it at time $t + \delay(t)$
with probability $1-p$. We assume that small messages (i.e., 1-bit
messages) can be sent reliably with negligible delay.  This assumption
is similar to the ``acknowledgments'' and ``permission'' messages used
in other works, see~\cite{XW-MDL:11,MG-DL-JSM-SD-KHJ:12} and
references therein.  Lastly, we also account for the possibility of
communication noise or quantization. We assume that messages among
agents are corrupted with an error which is upper bounded by some
$\maxnoise \geq 0$ known to the agents.

With this model, the \algosocial as described in
Algorithm~\ref{tab:social} does not guarantee convergence because the
monotonic behavior of the Lyapunov function no longer holds. The
problem occurs when an agent $j$ breaks a promise to agent $i$ at some
time $t$. If this occurs, agent $i$ will operate with invalid
information (due to the sources of error described above) and compute
$\si(\promisesets^i(t'))$ (as defined
in~\eqref{eq:upper-bound-Lie-derivative}) incorrectly for $t' \geq t$.

Next, we discuss how the \algosocial can be modified in scenarios with
unreliable communication.  To deal with communication noise, when an
agent $i$ receives an estimated promise $\hat{X}_j^i$ from another
agent $j$, it must be able to create a promise set $X_j^i$ that
contains the actual promise that agent $j$ intended to send.
We refer to this action as making a promise set valid. The following
example shows how it can be done for the promises described in
Remark~\ref{re:control-promises-rules}.

\begin{example}\longthmtitle{Ball-radius promise rule with
    communication noise} {\rm In the scenario with bounded
    communication noise, agent $j$ sends the control promise conveyed
    through $x_j(t), u_j(\promisesets^j(t)),$ and $\delta_j$, to agent
    $i$ at time $t$ as defined in
    Remark~\ref{re:control-promises-rules}, but $i$ receives instead
    $\hat{x}_j(t), \hat{u}_j(\promisesets^j(t)),$ and
    $\hat{\delta_j}$, where it knows that $\TwoNorm{x_j(t) -
      \hat{x}_j(t)} \leq \maxnoise$, $\TwoNorm{u_j(\promisesets^j(t))
      - \hat{u_j} (\promisesets^j(t))} \leq \maxnoise,$ and $|\delta_j
    - \hat{\delta}_j| \leq \bar{\delta}$, given that $\maxnoise$ and
    $\bar{\delta}$ are known a priori.  To ensure that the promise
    agent $i$ operates with about agent $j$ contains the true promise
    made by $j$, agent $i$ can set
    \begin{align*}
      U_j^i[t](t') = \cball{ \hat{u}_j^i(\promisesets^j(t))
      }{\hat{\delta}_j + \maxnoise + \bar{\delta}} \cap \UU_j \qquad
      t' \geq t .
    \end{align*}
    To create the state promise from this, $i$ would need the
    true state $x_j(t)$ of $j$ at time $t$. However, since only the
    estimate $\hat{x}_j^i(t)$ is available, 
    we modify~\eqref{eq:promise} by
    \begin{multline}\label{eq:promise-noise}
      X_j^i[t](t') = \cup_{y_j \in \cball{\hat{x}_j^i(t)}{\maxnoise} }
      \{ z \in \XX_j \; | \; \exists \, \map{u_j}{ [t, t'] }{\UU_j}
      \text{ with } u_j(s) \in U_j^i[t](s) \text{ for } s \in [t,t']
      \\
      \text{ such that } z = e^{A_j (t'-t)} y_j + \int_{t}^{t'}
      e^{A_j (t'-\tau)} B_j u_j(\tau) d\tau \} .  \eqoprocend
    \end{multline}
  }
\end{example}

We deal with the packet drops and communication delays with warning
messages similar to the ones introduced in
Section~\ref{se:brokenpromises}. 
Let an agent $j$ break its promise to agent $i$ at time $t$, then
agent $j$ sends~$i$ a new promise set $X_j^i[t]$ for $t' \geq t$ and
warning message WARN.  Since agent $i$ only receives WARN at time $t$,
the promise set $X_j^i[t]$ may not be available to agent $i$ for $t'
\geq t$. If the packet is dropped, then the message never comes
through, if the packet is successfully transmitted, then
$X_j^i[t](t')$ is only available for $t' \geq t + \delay(t)$. In
either case, we need a promise set $X_j^i[\cdot](t')$ for $t' \geq t$
that is guaranteed to contain $x_j(t')$. We do this by redefining the
promise using the reachable set, similarly
to~\eqref{eq:promise-delay}.  Note that this does not require the
agents to have a synchronized global clock, as the times $t'$ and $t$
are both monitored by the receiving agent $i$. In other words, it is
not necessary for the message sent by agent $j$ to be timestamped.  By
definition of reachable set, the promise $X_j^i[\cdot](t')$ is
guaranteed to contain $x_j(t')$ for $t' \geq t$.  If at time $t +
\maxdelay$, agent $i$ has still not received the promise $X_j^i[t]$
from~$j$, it can send agent $j$ a request REQ for a new message at
which point~$j$ would send~$i$ a new promise $X_j^i[t+\maxdelay]$.
Note that WARN is not sent in this case because the message was
requested from $j$ by~$i$ and not a cause of $j$ breaking a promise to
$i$. The \algosocialdelays, formally presented in
Algorithm~\ref{tab:socialdelays}, ensures the monotonic evolution of
the Lyapunov function $V$ even in the presence of packet drops,
communication delays, and communication noise.

\begin{algorithm}[htb!]
  {\footnotesize \vspace*{1ex} 
    \emph{(Self-trigger information update)}
    \\
    \parbox{\linewidth}{At any time $t$ agent $i \in \until{N}$
      receives new promise(s) $\hat{X}_j^i[t]$ from neighbor(s) $j \in
      \NN(i)$, agent $i$ performs:
    }\\[-4ex]
    \begin{algorithmic}[1]
      \STATE create valid promise $X_j^i[t]$ with respect to
      $\maxnoise$
      \STATE compute own control $\uteam_i(t')$ for $t' \geq t$
      using~\eqref{eq:agentsocialcontrol}
      \STATE compute own state evolution $x_i(t')$ for $t' \geq t$
      using~\eqref{eq:compute-own-evolution}
      \STATE compute first time $t^* \geq t$ such that $
      \si(\promisesets^i(t^*)) = 0 $
      \STATE schedule information request to neighbors in $\max \{ t^*
      - t, \dwellself \}$ seconds
      \WHILE {message from $j$ has not been received}
      \IF {current time equals $t + \max \{ t^* - t, \dwellself \} + k
        \maxdelay$ for $k \in \integernonnegative$} 
      \STATE send agent $j$ a request REQ for new information
      \ENDIF
      \ENDWHILE
    \end{algorithmic}
    \emph{(Respond to information request)}
    \\
    At any time $t$ a neighbor $j \in \NN(i)$
      requests information, agent $i$ performs:
    \begin{algorithmic}[1]
      \STATE send new promise $Y_i^j[t] =
      \staterule_i(\promisesets^i[\cdot]_{|[t,\infty)})$ to agent $j$ 
    \end{algorithmic}
    \emph{(Event-trigger information update)}
    \\
    At all times $t$, agent $i$ performs:
    \begin{algorithmic}[1]
      \IF {there exists $j \in \NN(i)$ such that $x_i(t) \notin
        Y_i^j[\cdot](t)$}
      \STATE send warning message WARN to agent $j$
      \IF {agent $i$ has sent a promise to $j$ at some time
        $\time{last} \in (t - \dwellevent, t]$}
      \STATE schedule to send new promise $Y_i^j[\time{last}+\dwellevent] =
      \staterule_i(\promisesets^i[\cdot]_{|[\time{last}+\dwellevent,\infty)})$
      to agent $j$ in $\time{last} + \dwellevent - t$ seconds
      \ELSE
      \STATE send new promise $Y_i^j[t] =
      \staterule_i(\promisesets^i[\cdot]_{|[t,\infty)})$ to agent $j$
      \ENDIF
      \ENDIF
    \end{algorithmic}
    \emph{(Respond to warning message)}
    \\
    \parbox{\linewidth}{ At any time $t$ agent $i \in \until{N}$
      receives a warning message WARN from agent $j \in \NN(i)$ }
    \\[-4ex]
    \begin{algorithmic}[1]
      \STATE redefine promise set $X_j^i[\cdot](t') = \cup_{x_j^0 \in
        X_j^i[\cdot](t)} \reachable_j(t'-t,x_j^0)$ for $t' \geq t$
      \WHILE {message from $j$ has not been received}
      \IF {current time equals $t + k \maxdelay$ for $k \in
        \integernonnegative$} 
      \STATE send agent $j$ a request REQ for new information
      \ENDIF
      \ENDWHILE
    \end{algorithmic}}
  \caption{\hspace*{-.5ex}: \algosocialdelays }\label{tab:socialdelays}
\end{algorithm}

The next result establishes the asymptotic correctness guarantees on
the \algosocialdelays. In the presence of communication noise or
delays, convergence can be guaranteed only to a set that contains the
desired set~$D$.

\begin{corollary}\label{cor:unreliable-convergence}
  Consider a networked cyber-physical system as described in
  Section~\ref{se:statement} with packet drops occurring with some
  unknown probability $0 \leq p < 1$, messages being delayed by some
  known maximum delay $\maxdelay$, and communication noise bounded
  by~$\maxnoise$, executing the \algosocialdelays
  (cf. Algorithm~\ref{tab:socialdelays}) with dwell times $\dwellself,
  \dwellevent > 0$ based on a continuous controller $u^{**} : \prod_{j
    \in \until{N}} \powerset( \XX_j ) \rightarrow \real^{m}$ that
  satisfies~\eqref{eq:social-assumptions} and is distributed over the
  communication graph~$\commgraph$.  Let
  \begin{align}\label{eq:robustconvergenceset}
    D'(\maxdelay,\maxnoise) = \setdef{x \in \XX}{ \inf_{
        {\availablestates^i}' \in \cball{
          \availablestates^i}{\maxnoise} } \si \Big( \{ x_i \} \times
        \prod_{j \in \NN(i)} \cup_{y_j \in \cball{ {x^i_j}'
          }{\maxnoise}} \reachable_j(\maxdelay, y_j)
      \Big) \geq 0 \\
      \text{ for all } i \in \until{N} } , \notag
  \end{align}
  Then, any bounded network trajectory with uniformly bounded promises
  asymptotically converges to $D'(\maxdelay,\maxnoise) \supset D$ with
  probability 1.
\end{corollary}
\begin{proof}
  We begin by noting that by
  equation~\eqref{eq:desiredderivativesocial}, the
  definition~\eqref{eq:upper-bound-Lie-derivative}, and the continuity
  of $u^{**}$, $D$ can be written as
  \begin{align*}
    D'(0,0) = \setdef{ x \in \XX }{ \sum_{i=1}^N \grad_i V(x) (A_i x_i
      + B_i u_i^{**}( \{ \availablestates^i \} )) \geq 0}.
  \end{align*}
  One can see that $D \subset D'(\maxdelay,\maxnoise)$ by noticing
  that, for any $x \in D$, $\maxnoise, \maxdelay \geq 0$, no matter
  which point ${\availablestates^i}' \in
  \cball{\availablestates^i}{\maxnoise}$ is taken, one has
  $\availablestates^i \in \{ x_i \} \times \prod_{j \in \NN(i)}
  \cup_{y_j \in \cball{ {x^i_j}' }{\maxnoise}} \reachable_j(\maxdelay,
  y_j)$.  To show that the bounded trajectories of the
  \algosocialdelays converge to~$D'$, we begin by noting that all
  properties of~$\algomap$ used in the proof of
  Proposition~\ref{prop:main-convergence} still hold in the presence
  of packet drops, delays, and communication noise as long as the time
  schedule~$\timeschedule^i$ is unbounded for each agent~$i \in
  \until{N}$.  In order for the time schedule $\timeschedule^i$ to be
  unbounded, each agent $i$ must receive an infinite number of
  messages, and $t_{\ell}^i \rightarrow \infty$.  Since packet drops
  have probability $0 \leq p < 1$, the probability that there is a
  finite number of updates for any given agent $i$ over an infinite
  time horizon is~$0$. Thus, with probability~$1$, there are an
  infinite number of information updates for each agent.  Using a
  similar argument to that of Lemma~\ref{le:zeno}, one can show that
  the positive dwell times $\dwellself, \dwellevent > 0$ ensure that
  Zeno behavior does not occur, meaning that $t_{\ell}^i \rightarrow
  \infty$. Then, by the analysis in the proof of
  Proposition~\ref{prop:main-convergence}, the bounded trajectories
  of~$\algomap$ still converge to $S^*$ as defined
  in~\eqref{eq:convergenceset}.
  
  For a bounded evolution $\gamma$ of the \algosocialdelays, we have
  that $\Omega(\gamma) \subset S^*$ is weakly positively invariant.
  Note that, since agents may never have exact information about their
  neighbors, we can no longer leverage
  properties~\eqref{eq:localassumptionsocial}
  and~\eqref{eq:desiredderivativesocial} to precisely
  characterize~$\Omega(\gamma)$.  We now show that for any $(Z,\ell)
  \in \Omega(\gamma)$, we have $\loc(Z,\ell) \in D'$.  Let $(Z,\ell)
  \in \Omega(\gamma)$. This means that $\si(Z^i_\NN) \geq 0$ for all
  $i \in \until{N}$. Take any agent $i$, by the \algosocialdelays,
  agent $i$ will request new information from neighbors in at most
  $\dwellself$ seconds.  This means there exists a state
  $(Z',\ell+\ell') \in \Omega(\gamma)$ for which agent $i$ has just
  received updated, possibly delayed, information from its neighbors
  $j \in \NN(i)$.  Since $(Z',\ell+\ell') \in S^*$, we know
  $\si({Z^i_\NN}') \geq 0$. We also know, since information was just
  updated, that ${Z_\NN^i}' \subset \{ {Z_i^i}' \} \times
    \prod_{j\in \NN(i)} \cup_{ y_j \in \cball{{z_j^i}'}{\maxnoise} }
    \reachable(\maxdelay, y_j)$.
 Since $(Z', \ell + \ell') \in S^*$,
  we know that $\si( {Z_\NN^i}' ) \geq 0$, for all $i \in
  \until{N}$. This means that $\loc(Z', \ell+\ell') \subset D'$, 
  thus $\loc(Z, \ell) \in S^* \subset D'$.
\end{proof}

From the proof of Corollary~\ref{cor:unreliable-convergence}, one can
see that the modifications made to the \algosocialdelays make the
omega limit sets of its trajectories larger than those of the
\algosocial, resulting in $D \subset D'$.  The set~$D'$ depends on the
Lyapunov function~$V$. However, the difference between
$D'(\maxdelay,\maxnoise)$ and~$D$ vanishes as $\maxnoise$ and
$\maxdelay$ vanish.

\myclearpage
\section{Simulations}\label{se:simulations}

In this section we present simulations of coordination strategies
derived from the team- and self-triggered approaches in a planar
multi-agent formation control problem.  Our starting point is the
distributed coordination algorithm based on graph rigidity analyzed
in~\cite{LK-MEB-BF:08,FD-BF:10} which makes the desired network
formation locally (but not globally) asymptotically stable. In this
regard, the state space $\XX$ of Section~\ref{se:statement}
corresponds to the domain of attraction of the desired equilibria and,
as long as the network trajectories do not leave this set, the
convergence results still hold.  The local convergence result of the
team-triggered approach here is only an artifact of the specific
example and, in fact, if the
assumptions~\eqref{eq:derivativeassumption} are satisfied globally,
then the system is globally asymptotically stabilized.  The interested
reader is referred to~\cite{CN-JC:13-sv} for a similar study in a
optimal networked deployment problem where the assumptions hold
globally.

Consider $4$ agents communicating over a graph which is only missing
the edge $(1,3)$ from the complete graph. The agents seek to attain a
rectangle formation of side lengths $1$ and $2$.  Each agent has
unicycle dynamics,
\begin{align*}
  \dot{x}_i &= u_i 
    \begin{bmatrix}
      \cos \theta_i
      \\
      \sin \theta_i 
    \end{bmatrix}
    \\
  \dot{\theta}_i &= v_i,
\end{align*}
where $0 \leq u_i \leq \umax = 5$ and $|v_i| \leq \vmax = 3$ are the
control inputs. The safe-mode controller is then simply
$(\safecontrol_i,v_i^\text{sf}) \equiv 0$.  The distributed control
law is defined as follows. Each agent computes a goal point
\begin{align*}\label{eq:graph-rigidity-law}
  p_i^*(x) & = x_i + \sum_{j \in \NN(i)} \left( \TwoNorm{ x_j - x_i } -
    d_{ij} \right) \unit{ x_j - x_i } ,
\end{align*}
where $d_{ij}$ is the pre-specified desired distance between agents
$i$ and $j$, and $\unit{x_j - x_i}$ denotes the unit vector in the
direction of $x_j - x_i$. Then, the control law is given by
\begin{align*}
  u_i^* &= \max \left\lbrace \min \{ k[\cos \theta_i \hspace*{2 mm}
    \sin \theta_i ]^T \cdot (p_i^*(x) - x_i), \umax \} , 0
  \right\rbrace,
  \\
  v_i^* &= \max \left\lbrace \min \{ k(\angle (p_i^*(x) - x_i) -
    \theta_i) , \vmax \} , -\vmax \right\rbrace ,
\end{align*}
where $k > 0$ is a design parameter. For our simulations we set
$k = 150$. This continuous control law
essentially ensures that the position $x_i$ moves towards $p_i^*(x)$
when possible while the unicycle rotates its orientation towards this
goal.  This control law ensures that $V: \left( \real^{2} \right)^N
\rightarrow \realnonnegative$ given~by
\begin{align*}
  V(x) = \frac{1}{2} \sum_{(i,j) \in E} \left( \TwoNorm{ x_j - x_i }^2
    - d_{ij}^2 \right)^2 ,
\end{align*}
is a nonincreasing function for the closed-loop system to establish the
asymptotic convergence to the desired formation.  For the
team-triggered approach, we use both static and dynamic ball-radius
promise rules. The controller~$\uteam$ is then defined
by~\eqref{eq:agentsocialcontrol}, where controller~$u^{**}$ is given
by~\eqref{eq:examplecontrol} as described in
Example~\ref{ex:estimate}. Note that although the agent has no forward
velocity when using the safe controller, it will still rotate in
place.  The initial conditions are $x_1(0) = (6,10)^T$, $x_2(0) =
(7,3)^T$, $x_3(0) = (14,8)^T$, and $x_4(0) = (7,13)^T$ and
$\theta_i(0) = \pi/2$ for all $i$. We begin by simulating the
team-triggered approach using fixed dwell times of $\dwellself = 0.3$
and $\dwellevent = 0.003$ and the static ball-radius promise of
Remark~\ref{re:control-promises-rules} with the same radius $\delta =
1$ for all agents. Figure~\ref{fig:trajectories} shows the
trajectories of the \algosocial.

{ \psfrag{onetwothree1}[cc][cc]{\hspace*{-3 mm} \footnotesize Agent
    1}%
  \psfrag{onetwothree2}[cc][cc]{\hspace*{-3 mm} \footnotesize Agent
    2}%
  \psfrag{onetwothree3}[cc][cc]{\hspace*{-3 mm} \footnotesize Agent
    3}%
  \psfrag{onetwothree4}[cc][cc]{\hspace*{-3 mm} \footnotesize Agent
    4}%
  \psfrag{onetwothreefourfivesix1}[cc][cc]{\hspace*{-3.5 mm}
    \footnotesize Team-triggered law}%
  \psfrag{onetwothreefourfivesix2}[cc][cc]{\hspace*{-8 mm}
    \footnotesize Self-triggered law}%
  \begin{figure}[htb]
    \centering
    \includegraphics[width=.45\linewidth]{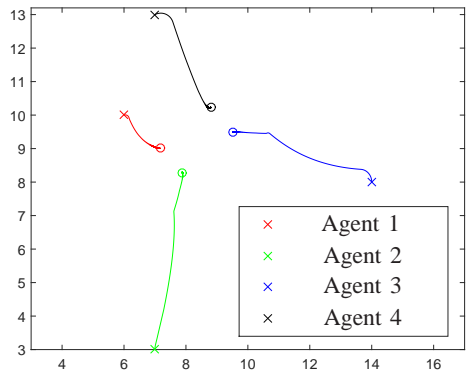}
    \caption{Trajectories of an execution of the \algosocial with
      fixed dwell times and promises. The initial and final
      condition of each agent is denoted by an `x' and an `o',
      respectively.}\label{fig:trajectories}
    \vspace{-1ex}
  \end{figure}
}

To compare the team- and self-triggered approaches, we denote by
$N_S^i$ the number of times~$i$ has requested new information (and
thus has received a message from each one of its neighbors) and by
$N_E^i$ the number of messages $i$ has sent to a neighboring agent
because it decided to break its promise. The total number of messages
for an execution is $N_\text{comm} = \sum_{i=1}^4 |\NN(i)| N_S^i +
N_E^i$.  Figure~\ref{fig:social}
compares 
the number of required communications in both approaches. Remarkably,
for this specific example, the team-triggered approach outperforms the
self-triggered approach in terms of required communication without
sacrificing any performance in terms of time to convergence (the
latter is depicted through the evolution of the Lyapunov function in
Figure~\ref{fig:social-adaptive}(b) below).  Less overall
communication has an important impact on reducing network load. In
Figure~\ref{fig:social}(a), we see that very quickly all agents are
requesting information as often as they can (as restricted by the
self-triggered dwell time), due to the conservative nature of the
self-triggered time computations. In the execution of the \algosocial
in Figure~\ref{fig:social}(b), we see that the agents are requesting
information from one another less frequently.
Figure~\ref{fig:social}(c) shows that agents were required to break a
few promises early on in the execution.
{ \psfrag{onetwo1}[cc][cc]{\hspace*{0 mm}
    \tiny Agent 1}%
  \psfrag{onetwo2}[cc][cc]{\hspace*{0 mm}
    \tiny Agent 2}%
  \psfrag{onetwo3}[cc][cc]{\hspace*{0 mm}
    \tiny Agent 3}%
  \psfrag{onetwo4}[cc][cc]{\hspace*{0 mm}
    \tiny Agent 4}%
   \psfrag{0}[cc][cc]{}%
   \psfrag{5}[cc][cc]{\scriptsize 5}
   \psfrag{10}[cc][cc]{\hspace*{-1 mm}\scriptsize 10}
   \psfrag{15}[cc][cc]{\hspace*{-0 mm}\scriptsize 15}
   \psfrag{20}[cc][cc]{\hspace*{-1 mm}\scriptsize 20}
   \psfrag{25}[cc][cc]{\hspace*{-0 mm}\scriptsize 25}
   \psfrag{30}[cc][cc]{\hspace*{-1 mm}\scriptsize 30}
   \psfrag{35}[cc][cc]{}
   \psfrag{35}[cc][cc]{\hspace*{-2 mm}\scriptsize 35}
   \psfrag{40}[cc][cc]{\hspace*{-1 mm}\scriptsize 40}
   \psfrag{45}[cc][cc]{\hspace*{-2 mm}\scriptsize 45}
   \psfrag{50}[cc][cc]{\hspace*{-1 mm}\scriptsize 50}
   \psfrag{55}[cc][cc]{}
   \psfrag{60}[cc][cc]{\hspace*{-1 mm}\scriptsize 60}
   \psfrag{65}[cc][cc]{}
   \psfrag{70}[cc][cc]{\hspace*{-1 mm}\scriptsize 70}
   \psfrag{80}[cc][cc]{\hspace*{-1 mm}\scriptsize 80}
   \psfrag{90}[cc][cc]{\hspace*{-1 mm}\scriptsize 90}
   \psfrag{100}[cc][cc]{\hspace*{-2 mm}\scriptsize 100}
   \psfrag{0.5}[cc][cc]{\hspace*{-1 mm}\scriptsize 0.5}
   \psfrag{1}[cc][cc]{\scriptsize 1}
   \psfrag{1.5}[cc][cc]{\hspace*{-1 mm}\scriptsize 1.5}
   \psfrag{2}[cc][cc]{\scriptsize 2}
   \psfrag{2.5}[cc][cc]{\hspace*{-1 mm}\scriptsize 2.5}
   \psfrag{3.5}[cc][cc]{\hspace*{-1 mm}\scriptsize 3.5}
   \psfrag{3}[cc][cc]{\scriptsize 3}
   \psfrag{4}[cc][cc]{\scriptsize 4}
   \psfrag{6}[cc][cc]{\scriptsize 6}
   \psfrag{7}[cc][cc]{\scriptsize 7}
   \psfrag{8}[cc][cc]{\scriptsize 8}
   \psfrag{9}[cc][cc]{\scriptsize 9}
   \psfrag{0}[cc][cc]{\scriptsize 0}
  \begin{figure}[htb]
    \centering
    \subfigure[]{\includegraphics[width=.32\linewidth]{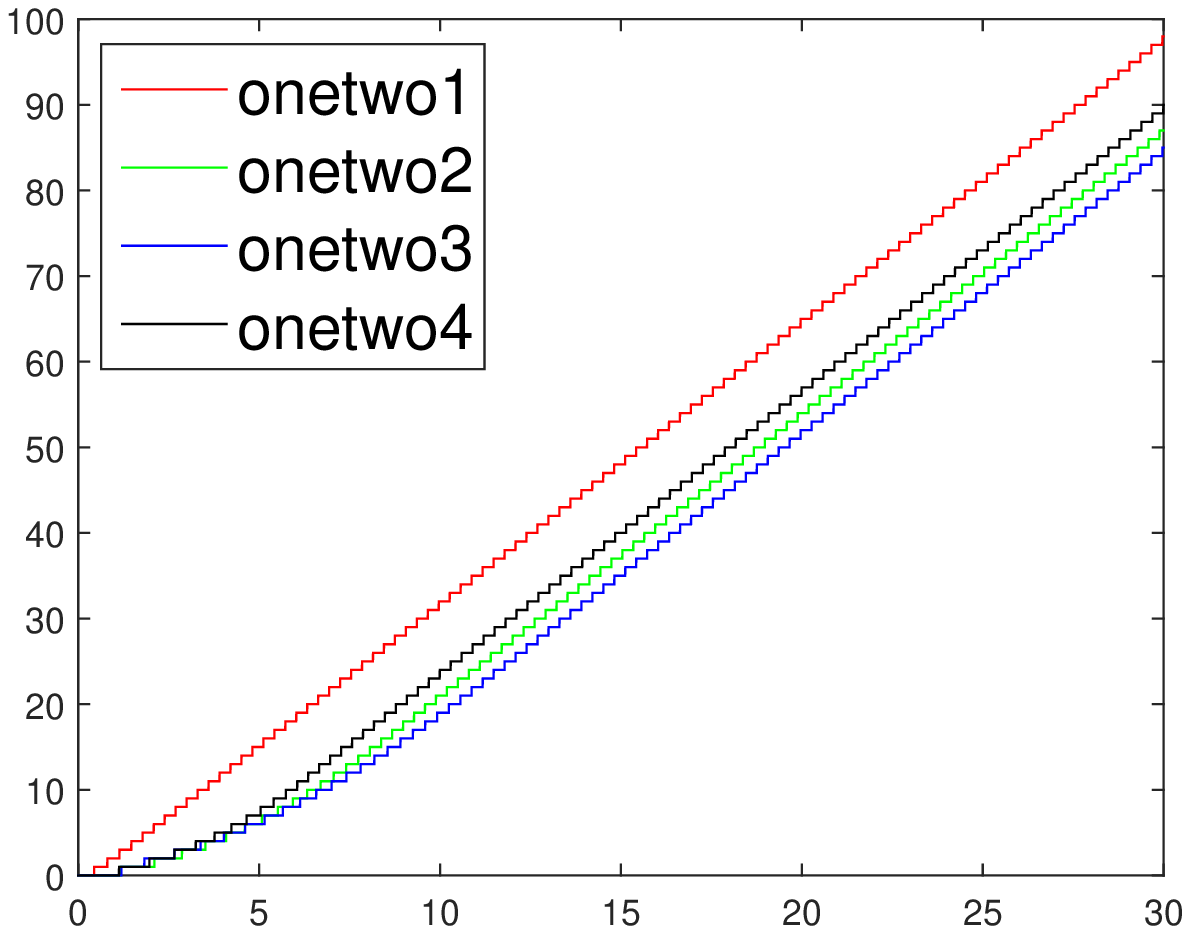}}
    \put(-158,65){\footnotesize $N_S$} \put(-80,-5){{\footnotesize
        Time}}
    \;
    \subfigure[]{\includegraphics[width=.32\linewidth]{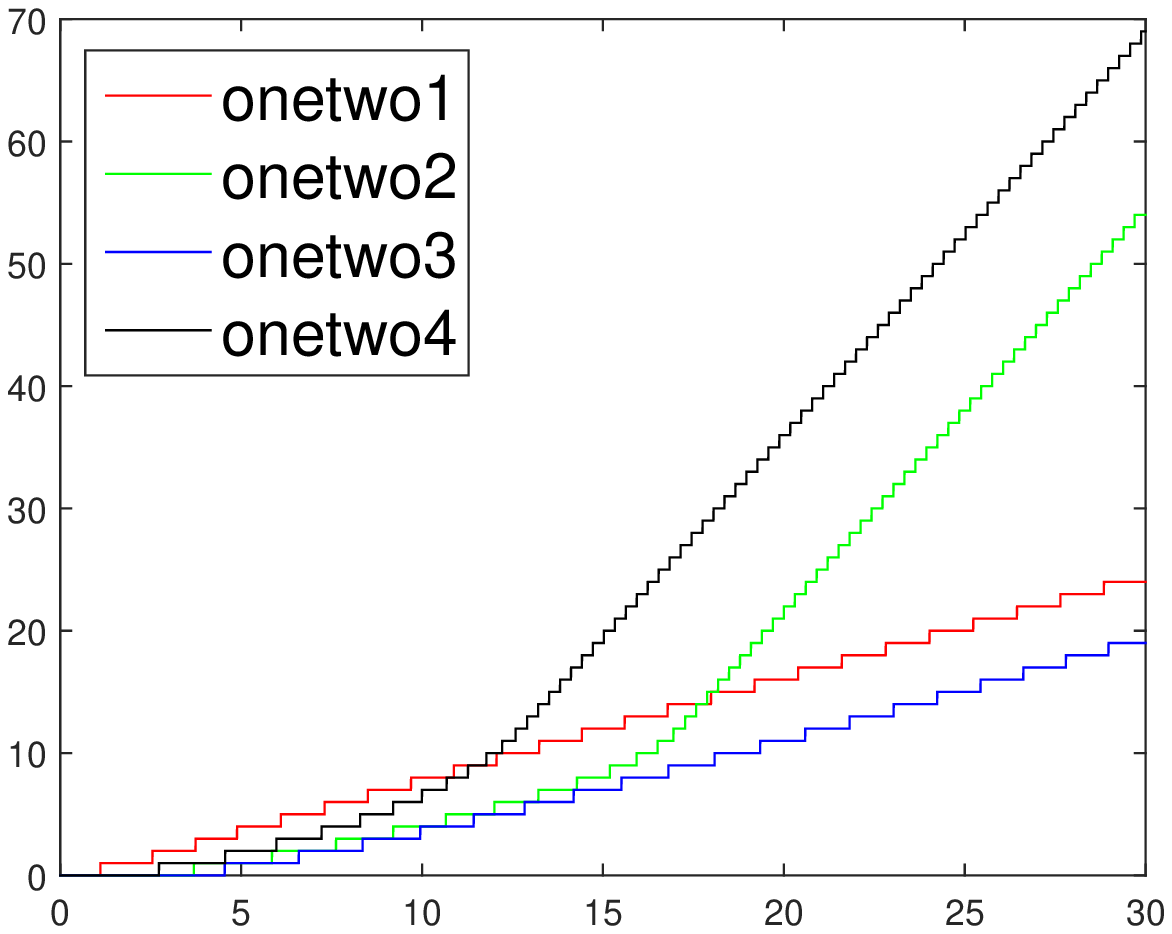}}
    \put(-158,65){\footnotesize $N_S$} \put(-80,-5){{\footnotesize
        Time}} 
    \;
    \subfigure[]{\includegraphics[width=.32\linewidth]{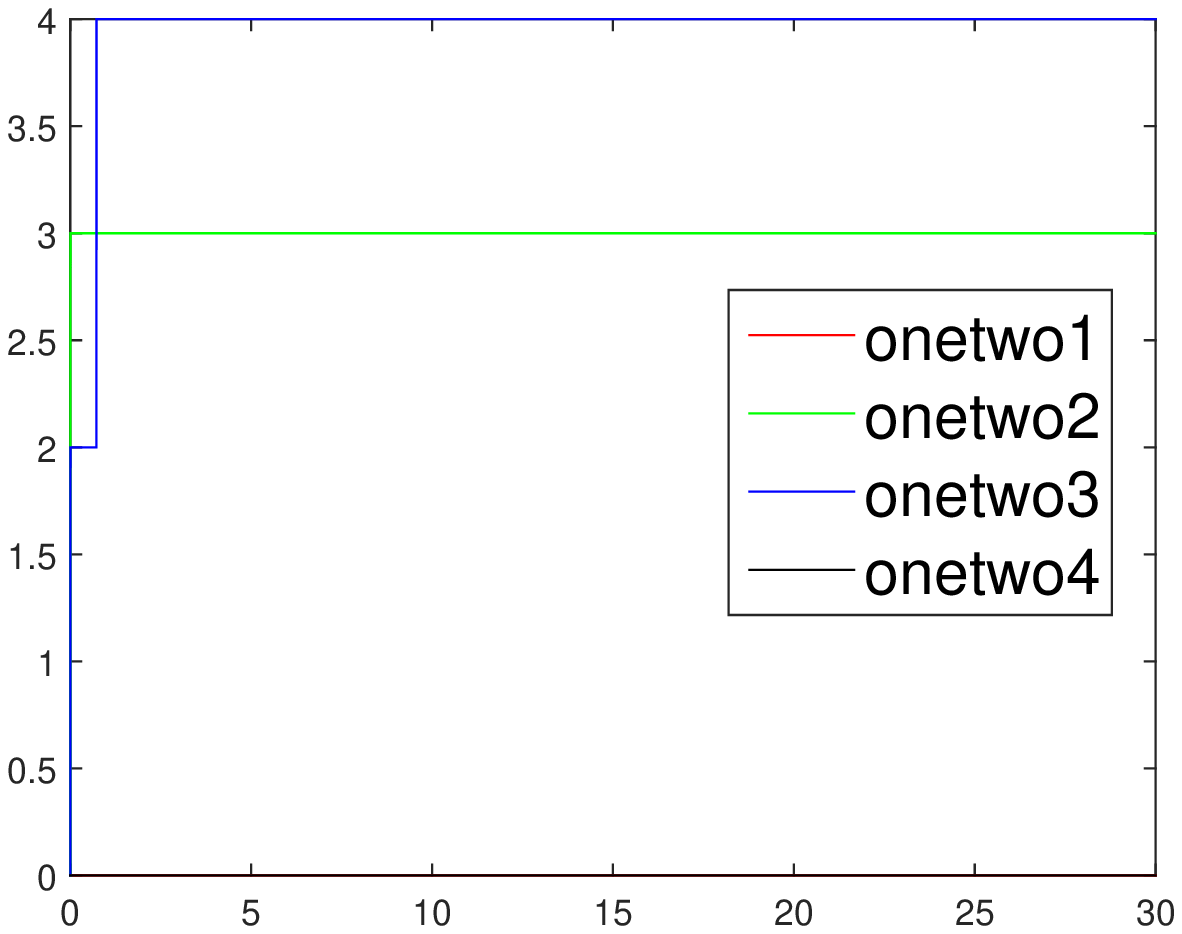}}
    \put(-158,65){\footnotesize $N_E$} \put(-80,-5){{\footnotesize Time}}
    \caption{Number of self-triggered requests made by each agent in
      an execution of the (a) self-triggered approach and (b)
      team-triggered approach with fixed dwell times and promises. For
      the latter execution, (c) depicts the number of event-triggered
      messages sent (broken promises) by each
      agent.}\label{fig:social}
    \vspace{-1ex}
  \end{figure}
}
{  
   \psfrag{0.5}[cc][cc]{\hspace*{-1 mm}\scriptsize 0.5}
   \psfrag{0.45}[cc][cc]{\hspace*{-1 mm}\scriptsize 0.45}
   \psfrag{0.4}[cc][cc]{\hspace*{-1 mm}\scriptsize 0.4}
   \psfrag{0.35}[cc][cc]{\hspace*{-1 mm}\scriptsize 0.35}
   \psfrag{0.3}[cc][cc]{\hspace*{-1 mm}\scriptsize 0.3}
   \psfrag{0.25}[cc][cc]{\hspace*{-1 mm}\scriptsize 0.25}
   \psfrag{0.2}[cc][cc]{\hspace*{-1 mm}\scriptsize 0.2}
   \psfrag{0.15}[cc][cc]{\hspace*{-1 mm}\scriptsize 0.15}
   \psfrag{0.1}[cc][cc]{\hspace*{-1 mm}\scriptsize 0.1}
   \psfrag{0.05}[cc][cc]{\hspace*{-1 mm}\scriptsize 0.05}
   \psfrag{0.6}[cc][cc]{\hspace*{-1 mm}\scriptsize 0.6}
   \psfrag{0.8}[cc][cc]{\hspace*{-1 mm}\scriptsize 0.8}
   \psfrag{100}[cc][cc]{\hspace*{-2 mm}\scriptsize 100}
   \psfrag{300}[cc][cc]{\hspace*{-2 mm}\scriptsize 300}
   \psfrag{200}[cc][cc]{\hspace*{-2 mm}\scriptsize 200}
   \psfrag{500}[cc][cc]{\hspace*{-2 mm}\scriptsize 500}
   \psfrag{600}[cc][cc]{\hspace*{-2 mm}\scriptsize 600}
   \psfrag{700}[cc][cc]{\hspace*{-2 mm}\scriptsize 700}
   \psfrag{400}[cc][cc]{\hspace*{-2 mm}\scriptsize 400}
   \psfrag{900}[cc][cc]{\hspace*{-2 mm}\scriptsize 900}
   \psfrag{500}[cc][cc]{\hspace*{-2 mm}\scriptsize 500}
   \psfrag{800}[cc][cc]{\hspace*{-2 mm}\scriptsize 800}
   \psfrag{1000}[cc][cc]{\hspace*{-2 mm}\scriptsize 1000}
   \psfrag{1200}[cc][cc]{\hspace*{-2 mm}\scriptsize 1200}
   \psfrag{1}[cc][cc]{\vspace*{-4 mm}\scriptsize 1}
   \psfrag{1.5}[cc][cc]{\vspace*{-4 mm}\scriptsize 1.5}
   \psfrag{0}[cc][cc]{\scriptsize 0}
\begin{figure}[htb]
  \centering
  \subfigure[]{\includegraphics[width=.45\linewidth]{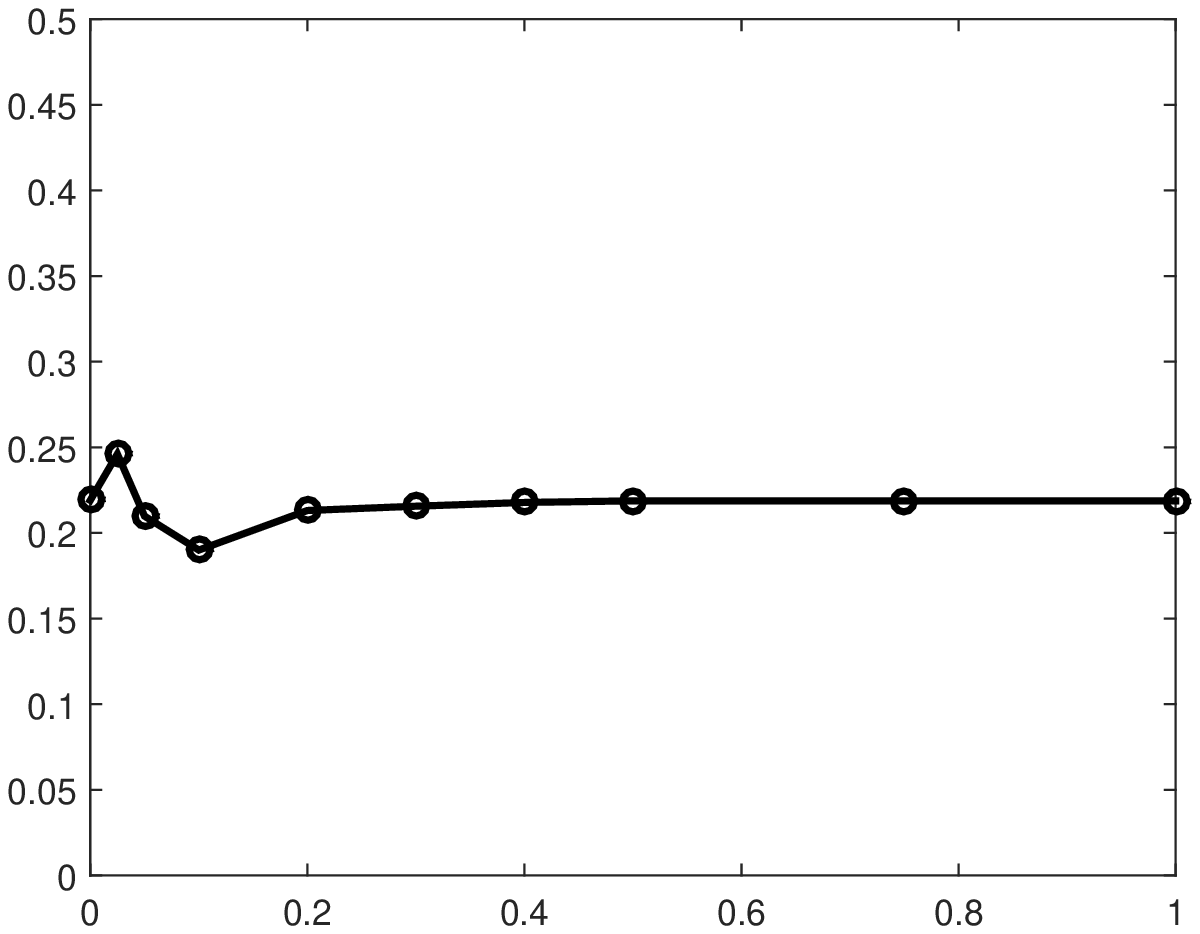}}
  \put(-225,105){\footnotesize $V(30)$} \put(-107,5){{\footnotesize $\lambda$}}
  \quad
  \subfigure[]{\includegraphics[width=.45\linewidth]{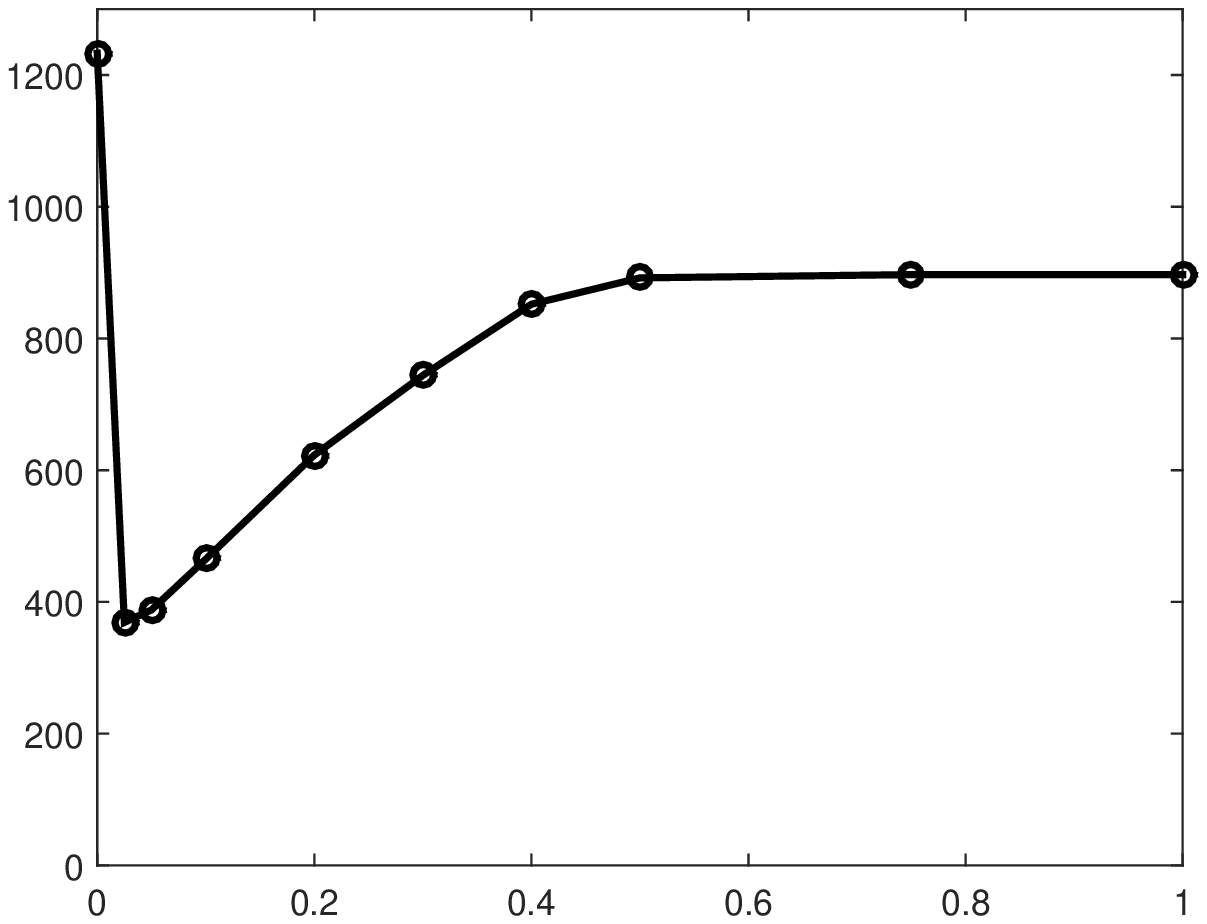}}
  \put(-222,105){\footnotesize $N_\text{comm}$} \put(-107,5){{\footnotesize $\lambda$}}
  \caption{Plots of (a) the value of the Lyapunov function at a fixed
    time (30 sec) and (b) the total number of messages exchanged in
    the network by this time for the team-triggered approach with
    varying tightness of promises
    $\lambda$.}\label{fig:changinglambda}
  \vspace{-1ex}
\end{figure}
}

Next, we illustrate the role that the tightness of promises has on the
network performance. With the notation of
Remark~\ref{re:control-promises-rules} for the static ball-radius
rule, let $\lambda = \frac{\delta}{2 \umax}$. Note that when $\lambda
= 0$, the promise generated by~\eqref{eq:staticpromise} is a
singleton, i.e., an exact promise. On the other hand, when $\lambda =
1$, the promise generated by~\eqref{eq:staticpromise} contains the
reachable set, corresponding to no actual commitment being made (i.e.,
the self-triggered approach
).  Figure~\ref{fig:changinglambda} compares the value of the Lyapunov
function after a fixed amount of time (30 seconds) and the total number
of messages sent $N_\text{comm}$ between agents by this time for
varying tightness of promises. The dwell times here are fixed at
$\dwellself = 0.3$ and $\dwellevent = 0.003$.  Note that a suitable
choice of~$\lambda$ helps greatly reduce the amount of communication
compared to the self-triggered approach ($\lambda = 1$) while
maintaining a similar convergence rate.

Finally, we demonstrate the added benefits of using adaptive promises
and dwell times.  Figure~\ref{fig:social-adaptive}(a) compares the
total number of messages sent in the self-triggered approach and the
team-triggered approaches with fixed promises and dwell times (FPFD),
fixed promises and adaptive dwell times (FPAD), adaptive promises and
fixed dwell times (APFD), and adaptive promises and dwell times
(APAD).  The parameters of the adaptive dwell time used
in~\eqref{eq:adaptive-dwell} are $\delta_d = 0.15$ and $\Delta_d =
0.3$. For agent $j \in \until{4}$, the radius~$\delta_j$ of the
dynamic ball-radius rule of Remark~\ref{re:control-promises-rules} is
$\delta_j (t) = 0.50 \TwoNorm{ u^{**}_j( \promisesets^j(t) ) -
  \safecontrol_j(x_j(t)) } + 10^{-6}$.  This plot shows the advantage
of the team-triggered approach in terms of required communication over
the self-triggered one and also shows the additional benefits of
implementing the adaptive promises and dwell time. This is because by
using the adaptive dwell time, agents decide to wait longer periods
for new information while their neighbors are still moving. By using the
adaptive promises, as agents near convergence, they are able to make
increasingly tighter promises, which allows them to request
information from each other less frequently. As
Figure~\ref{fig:social-adaptive}(b) shows, the network performance is
not compromised despite the reduction in communication.

{ \psfrag{onetwothreefour1}[cc][cc]{\hspace*{-10 mm} \scriptsize
    Self-triggered}%
  \psfrag{onetwothreefour2}[cc][cc]{\hspace*{-1.5 mm} \scriptsize
    Team-triggered FPFD}%
  \psfrag{onetwothreefour3}[cc][cc]{\hspace*{-1.5 mm} \scriptsize
    Team-triggered FPAD}%
  \psfrag{onetwothreefour4}[cc][cc]{\hspace*{-1.5 mm} \scriptsize
    Team-triggered APFD}%
  \psfrag{onetwothreefour5}[cc][cc]{\hspace*{-1.5 mm} \scriptsize
    Team-triggered APAD}%
  \psfrag{onetwoth4}[cc][cc]{\hspace*{0 mm} \scriptsize Agent 4}%
  \psfrag{0}[cc][cc]{}%
   \psfrag{10}[cc][cc]{\hspace*{-1 mm}\scriptsize 10}
   \psfrag{5}[cc][cc]{\hspace*{-1 mm}\scriptsize 5}
   \psfrag{15}[cc][cc]{\hspace*{-1 mm}\scriptsize 15}
   \psfrag{25}[cc][cc]{\hspace*{-1 mm}\scriptsize 25}
   \psfrag{20}[cc][cc]{\hspace*{-1 mm}\scriptsize 20}
   \psfrag{30}[cc][cc]{\hspace*{-1 mm}\scriptsize 30}
   \psfrag{40}[cc][cc]{\hspace*{-1 mm}\scriptsize 40}
   \psfrag{50}[cc][cc]{\hspace*{-1 mm}\scriptsize 50}
   \psfrag{60}[cc][cc]{\hspace*{-1 mm}\scriptsize 60}
   \psfrag{70}[cc][cc]{\hspace*{-1 mm}\scriptsize 70}
   \psfrag{80}[cc][cc]{\hspace*{-1 mm}\scriptsize 80}
   \psfrag{90}[cc][cc]{\hspace*{-1 mm}\scriptsize 90}
   \psfrag{100}[cc][cc]{\hspace*{-2 mm}\scriptsize 100}
   \psfrag{150}[cc][cc]{\hspace*{-2 mm}\scriptsize 150}
   \psfrag{200}[cc][cc]{\hspace*{-2 mm}\scriptsize 200}
   \psfrag{250}[cc][cc]{\hspace*{-2 mm}\scriptsize 250}
   \psfrag{300}[cc][cc]{\hspace*{-2 mm}\scriptsize 300}
   \psfrag{350}[cc][cc]{\hspace*{-2 mm}\scriptsize 350}
   \psfrag{400}[cc][cc]{\hspace*{-2 mm}\scriptsize 400}
   \psfrag{450}[cc][cc]{\hspace*{-2 mm}\scriptsize 450}
   \psfrag{500}[cc][cc]{\hspace*{-2 mm}\scriptsize 500}
   \psfrag{600}[cc][cc]{\hspace*{-2 mm}\scriptsize 600}
   \psfrag{700}[cc][cc]{\hspace*{-2 mm}\scriptsize 700}
   \psfrag{400}[cc][cc]{\hspace*{-2 mm}\scriptsize 400}
   \psfrag{900}[cc][cc]{\hspace*{-2 mm}\scriptsize 900}
   \psfrag{500}[cc][cc]{\hspace*{-2 mm}\scriptsize 500}
   \psfrag{800}[cc][cc]{\hspace*{-2 mm}\scriptsize 800}
   \psfrag{1000}[cc][cc]{\hspace*{-2 mm}\scriptsize 1000}
   \psfrag{0.5}[cc][cc]{\vspace*{-4 mm}\scriptsize 0.5}
   \psfrag{1}[cc][cc]{\vspace*{-4 mm}\scriptsize 1}
   \psfrag{1.5}[cc][cc]{\vspace*{-4 mm}\scriptsize 1.5}
   \psfrag{0}[cc][cc]{\scriptsize 0}
  \begin{figure}[htb]
    \centering
    \subfigure[]{\includegraphics[width=.45\linewidth]{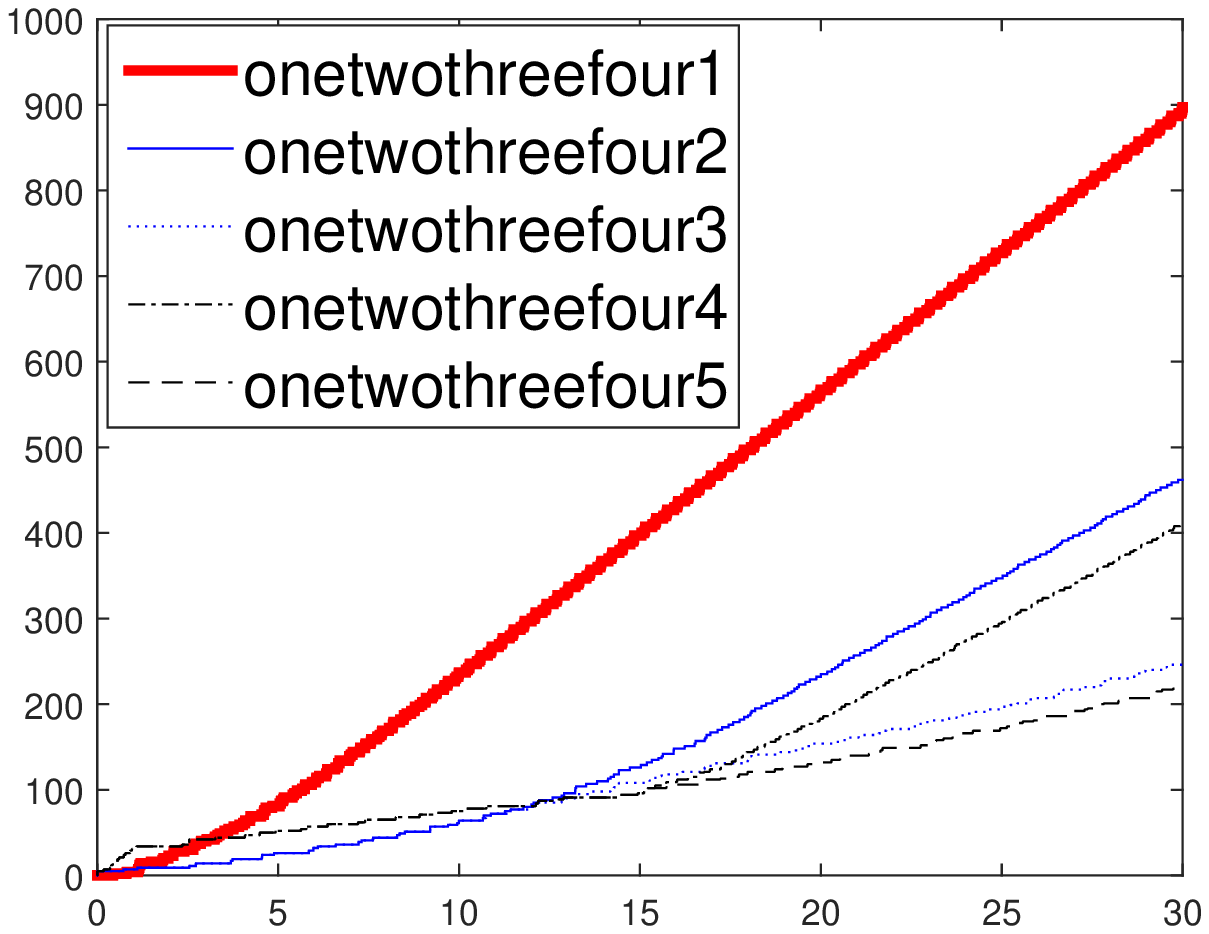}}
    \put(-225,105){\small $N_\text{comm}$}
    \put(-110,4){{\footnotesize Time}}
    \quad
    \subfigure[]{\includegraphics[width=.45\linewidth]{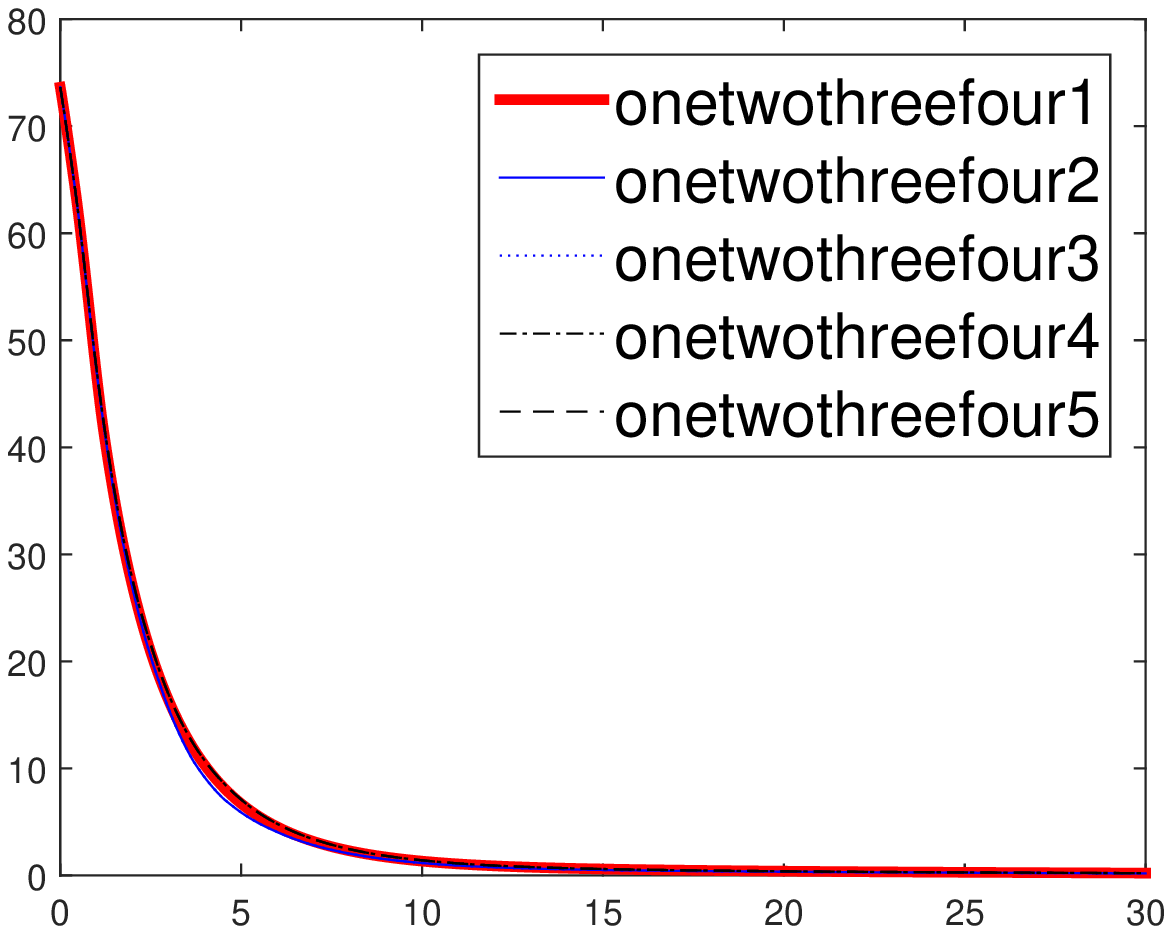}}
    \put(-205,105){\small $V$} \put(-110,4){{\footnotesize Time}}
    \caption{Plots of (a) the total number of messages sent and (b)
      the evolution of the Lyapunov function~$V$ for executions of
      self-triggered approach and the team-triggered approaches with
      fixed promises and dwell times (FPFD), fixed promises and
      adaptive dwell times (FPAD), adaptive promises and fixed dwell
      times (APFD), and adaptive promises and dwell
      times~(APAD). }\label{fig:social-adaptive}
    \vspace*{-1ex}
\end{figure}
}

\myclearpage
\section{Conclusions}\label{se:conclusions}

We have proposed a novel approach, termed team-triggered, that
combines ideas from event- and self-triggered control for the
implementation of distributed coordination strategies for networked
cyber-physical systems.  Our approach is based on agents making
promises to each other about their future states.  If a promise is
broken, this triggers an event where the corresponding agent provides
a new commitment. As a result, the information available to the agents
is set-valued and can be used to schedule when in the future further
updates are needed.  We have provided a formal description and
analysis of team-triggered coordination strategies and have also
established robustness guarantees in scenarios where communication is
unreliable. The proposed approach opens up numerous venues for future
research. Among them, we highlight the robustness under disturbances
and sensor noise, more general models for individual agents, the
design of team-triggered implementations that guarantee the invariance
of a desired set in distributed scenarios, the relaxation of the
availability of the safe-mode control via controllers that allow
agents to execute maneuvers that bring them back to their current
state, relaxing the requirement on the negative semidefiniteness of
the derivative of the Lyapunov function along the evolution of each
individual agent, methods for the systematic design of controllers
that operate on set-valued information models, understanding the
implementation trade-offs in the design of promise rules, analytic
guarantees on the performance improvements with respect to
self-triggered strategies, and the impact of evolving topologies on
the generation of promises.


\end{document}